\documentclass[letterpaper]{article}

\usepackage{amsmath}
\usepackage{amsfonts}
\usepackage{amsthm}
\usepackage{amssymb}
\usepackage{datetime}
\usepackage{color}

\theoremstyle{plain}
\newtheorem{theorem}{Theorem}[section]
\newtheorem{lemma}[theorem]{Lemma}
\newtheorem{corollary}[theorem]{Corollary}
\theoremstyle{definition}
\newtheorem{definition}[theorem]{Definition}
\newtheorem{problem}[theorem]{Problem}
\newtheorem{example}[theorem]{Example}
\newtheorem{remark}[theorem]{Remark}
\newtheorem{notation}[theorem]{Notation}

\numberwithin{equation}{section}

\newcommand{\all}{\hbox{for all}}
\newcommand{\All}{\hbox{For all}}
\newcommand{\bra}[2]{\langle#1,#2\rangle}

\newcommand{\bigcupn}{\bigcup\nolimits}
\newcommand{\Bra}[2]{\big\langle#1,#2\big\rangle}

\newcommand{\dbs}{^{**}}
\newcommand{\dist}{\hbox{\rm dist}}
\newcommand{\dom}{\hbox{\rm dom}}
\newcommand{\eighth}{\ts\frac{1}{8}}

\newcommand{\eps}{\varepsilon}

\newcommand{\F}{{\mathbb F}}
\newcommand{\fourth}{\ts\frac{1}{4}}

\newcommand{\half}{{\textstyle\frac{1}{2}}}
\newcommand{\hunder}{\underbar{$h$}}
\newcommand{\kunder}{\underbar{$k$}}
\newcommand{\I}{\mathbb I}
\newcommand{\ifff}{\Longleftrightarrow}
\newcommand{\infn}{\inf\nolimits}
\newcommand{\intr}{\hbox{\rm int}}

\newcommand{\limn}{\lim\nolimits}

\newcommand{\limsupn}{\limsup\nolimits}
\newcommand{\lin}{\hbox{\rm lin}}
\newcommand{\lr}{\Longrightarrow}
\newcommand{\LT}{{\wt L}}

\newcommand{\minn}{\min\nolimits}
\newcommand{\on}{\hbox{on}}
\newcommand{\PC}{{\cal PC}}
\newcommand{\PCLSC}{{\cal PCLSC}}
\newcommand{\PCLSCQ}{{\cal PCLSC}_q}
\newcommand{\PCQ}{{\cal PC}_q}

\newcommand{\qlr}{\quad\Longrightarrow\quad}
\newcommand{\qLt}{q_{\wt L}}

\newcommand{\quand}{\quad\hbox{and}\quad}
\newcommand{\rbar}{\,]{-}\infty,\infty]}

\newcommand{\RR}{\mathbb R}
\newcommand{\rl}{\Longleftarrow}
\newcommand{\rLt}{r_{\wt L}}

\newcommand{\sqs}{^{*@}}

\newcommand{\supn}{\sup\nolimits}
\newcommand{\T}{{\bf T}}
\newcommand{\toto}{\rightrightarrows}
\newcommand{\ts}{\textstyle}

\newcommand{\wh}{\widehat}
\newcommand{\wt}{\widetilde}

\newcommand{\xbra}[2]{\lfloor#1,#2\rfloor}
\newcommand{\Cor}{Corollary~\ref}
\newcommand{\Def}{Definition~\ref}
\newcommand{\Defs}{Definitions~\ref}
\newcommand{\Ex}{Example~\ref}
\newcommand{\Lem}{Lemma~\ref}
\newcommand{\Lems}{Lemmas~\ref}

\newcommand{\Rem}{Remark~\ref}
\newcommand{\Sec}{Section~\ref}
\newcommand{\Secs}{Sections~\ref}
\newcommand{\Thm}{Theorem~\ref}
\newcommand{\Thms}{Theorems~\ref}

\title{``Densities'' and maximal monotonicity I}
\author{
Stephen Simons
\thanks{
Department of Mathematics, University of California, Santa Barbara, CA\ 93106-3080, U.S.A.
Email: \texttt{stesim38@gmail.com}.}}

\date{}

\begin{document}

\maketitle

\begin{abstract}\noindent
We discuss ``Banach SN spaces'', which include Hilbert spaces, negative Hilbert spaces, and the product of any real Banach space with its dual. We introduce ``L-positive'' sets, which generalize monotone multifunctions from a Banach space into its dual. We introduce the concepts of ``$r_L$--density'' and its specialization ``quasidensity'': the closed quasidense monotone multifunctions from a Banach space into its dual form a (generally) strict subset of the maximally monotone ones, though all surjective maximally monotone and all maximally monotone multifunctions on a reflexive space are quasidense.  We give a sum theorem and a parallel sum theorem for closed monotone quasidense multifunctions under very general constraint conditions.   That is to say, quasidensity obeys a very nice calculus rule.   We give a short proof that the subdifferential of a proper convex lower semicontinuous function on a Banach space is quasidense, and deduce generalizations of the Brezis--Browder theorem on linear relations to non reflexive Banach spaces. We also prove that any closed monotone quasidense multifunction has a number of other very desirable properties.
\end{abstract}




\section{Introduction}
\setlength{\belowdisplayskip}{4pt}
\setlength{\belowdisplayshortskip}{4pt}
\setlength{\abovedisplayskip}{4pt}
\setlength{\abovedisplayshortskip}{4pt}
This paper falls logically into four parts.   In the first part, \Secs{SNsec}--\ref{PHIsec}, \ref{LINsec} and part of \Sec{ZAGsec}, we discuss ``Banach SN spaces'', ``$L$--positive sets'',\break ``$r_L$--density'', ``touching functions'' and the functions $\Phi_A$ and $\Theta_A$ determined by an $L$--positive set.   The second part, \Secs{EEsec}--\ref{SUMSsec}, \ref{LINMONsec} and \ref{FITZCHARsec}, is about Banach SN spaces of the special form $E \times E^*$, where $E$ is a nonzero real Banach space.   This part includes a short proof of a strict generalization of Rockafellar's result (see \cite{RTRMMT}) that the subdifferential of a proper convex lower semicontinuous function on a Banach space is maximally monotone, a sum theorem, a parallel sum theorem, and generalizations of the Brezis--Browder theorem on monotone linear relations.     The third part, \Sec{BRsec}, is about negative alignment conditions.   The fourth part, \Sec{ZAGsec}, is about a generalization of an inequality of Zagrodny.
\par
We now give an overview of the second part of this paper in which the notation will be familiar to the greatest number of readers, namely \Secs{EEsec}--\ref{SUMSsec}, \ref{LINMONsec} and \ref{FITZCHARsec}.   We give the initial definitions in \Sec{EEsec}.   The rather cumbersome definition of $r_L$--density in this special situation appears in \eqref{EE1}.   We use the term ``quasidensity'' instead of ``$r_L$--density'' in this context.   There are two other conditions equivalent to quasidensity in this paper, which can be found in\break \Thm{OORthm} and \Thm{NIBRthm}.   More can be found in \cite[Theorems 13.3, 13.6, 15.2 and 18.5.]{RLv7}
\par
It is shown in \Thm{RLMAXthm}(a) that a closed, monotone, quasidense set is maximally monotone.   \Thm{RTRthm}, \Cor{SURJcor} and \Thm{RLMAXthm}(b) show that  the closed, monotone, quasidense multifunctions do not form {\em too} small a class of sets.   We now discuss these three results.
\par
In \Thm{RTRthm}, we prove that the subdifferential of a proper, convex, lower semicontinuous function is quasidense.   The main nontrivial building blocks in the proof of \Thm{RTRthm} are Rockafellar's formula for the conjugate of a sum and the Cauchy sequence argument used in the proof of \Thm{RLEQthm}\big((b)$\lr$(c)\big).    By contrast with the proof given in \Thm{RTRthm}, it also is possible to give a ``direct proof''  using none of the results of \Sec{RLsec} after \Def{RLdef}, but using instead a separation theorem in $E \times E^*$, the Br{\o}ndsted--Rockafellar\break theorem and Rockafellar's formula for the subdifferential of a sum.   This ``direct proof'' is not much harder than the most recent proof of Rockafellar's original result that we have seen in print.   Since the formula for the subdifferential of a sum is very close to the formula for the conjugate of a sum, this leads one to speculate that the proof of \Thm{RLEQthm}\big((b)$\lr$(c)\big) is somehow related to the Br{\o}ndsted--Rockafellar theorem (or, more likely, Ekeland's variational principle).   We refer the reader to \cite[Theorem 8.4, p. 15]{RLv1} for more details of this ``direct proof''.   We do not discuss it any more in this paper because the result established in Simons--Wang, \cite[Theorem 3.2]{SW} shows that the (appropriate) subdifferential of an (appropriate) proper (and not necessarily convex) lower semicontinuous function is quasidense.   In other word, $r_L$--density and quasidensity have some interest outside the context of $L$--positive or monotone sets, as the case may be.
\par
Finally, we mention the novel use of \Thm{RTRthm} in \Lem{LINlem} to obtain results on linear sets (see below).   We do not know if there are similar applications of \Thm{RTRthm} to the nonlinear case.   It would be very intriguing if there were such applications.
\par
In \Cor{SURJcor} we prove that every surjective maximally monotone\break multifunction is quasidense, and in \Thm{RLMAXthm}(b) we prove that if $E$ is\break reflexive then every maximally monotone multifunction on $E$ is quasidense.  
\par
A useful counterexample for the nonsurjective, nonreflexive or non--\break subdifferential case is the {\em tail operator} (see \Ex{TAILex}), which is a\break maximally monotone linear operator from $E = \ell_1$ into $\ell_\infty = E^*$ that is not quasidense.   This example brings into stark relief the difference of behavior between surjective maximally monotone multifunctions and maximally monotone multifunctions with full domain.   
\par
\Sec{SUMSsec} is devoted to a sum theorem and a parallel sum theorem for closed, monotone quasidense multifunctions:   \Thm{STDthm} contains a result that implies that if $S$ and $T$ are closed monotone quasidense multifunctions and the effective domains $D(S)$ and $D(T)$ satisfy the Rockafellar constraint condition then\break $S + T$ is closed, monotone and quasidense.   \Thm{STRthm} contains an analogous but more technical result when we have information about the ranges $R(S)$ and $R(T)$.   Since closed, monotone, quasidense monotone multifunctions are maximally monotone, \Thm{STDthm} presents a stark contrast to the situation for maximally monotone multifunctions: it is still apparently not known whether the sum of two maximally monotone multifunctions satisfying the Rockafellar constraint condition is maximally monotone.   \Thm{STRthm} uses the concept of the ``Fitzpatrick extension'' of a closed, monotone quasidense multifunction, which is defined in \Def{FITZdef}, and further developed in \Sec{FITZCHARsec}.
\par
The quasidensity of subdifferentials is used in \Sec{LINsec} to obtain results about closed {\em linear} $L$--positive subspaces of Banach SN spaces.   These results are\break applied in \Sec{LINMONsec} to  monotone linear relations.   Specifically, it is proved in \Thm{LINMONthm} that if $A$ is a closed monotone linear relation with adjoint relation $A^\T$ then $A$ is quasidense if, and only if, $A^\T$ is monotone if, and only if $A^\T$ is maximally monotone.   This extends results established in \cite{BBWYLIN} and \cite{BBWYBB} by Bauschke, Borwein, Wang and Yao for general Banach spaces which, in turn,\break extend a result proved by Brezis and Browder in \cite{BB} for reflexive Banach spaces.   It is also worthy of note that \Thm{HTthm} provides a two--dimensional quadrant of examples of maximally monotone linear operators that fail to be quasidense.
\par
It is interesting to observe that the analysis of \Secs{LINsec} and \ref{LINMONsec} uses the quasidensity of subsets of $B \times B^* = (E \times E^*) \times (E \times E^*)^* = E \times E^* \times E^* \times E\dbs$.   The relatively simple notation seems to hide the actual complexity of the objects being considered.
\par
We now discuss the third part of the paper, \Sec{BRsec}, in which we discuss negative alignment conditions. \Thm{MFBRthm}(c) contains a version of the Br{\o}ndsted--Rockafellar theorem for closed, monotone, quasidense sets extending part of \cite[Theorem~4.2]{ASBR}. In \Thm{MFBRthm}(d), we prove that the projections of such a sets have convex closures. Finally, in \Thm{NIBRthm}, we give a criterion in terms of negative alignment for a closed monotone set to be quasidense.   In \Def{ANAdef}, we define
maximal monotonicity of ``type (ANA)''.   This concept was introduced, though unnamed, in \cite[Theorem~4.5, pp.\ 367--369]{SAM} as a property of subdifferentials.   We prove in \Thm{ANAthm} that a closed monotone quasidense set is maximally monotone of type (ANA).   We do not have an example of a maximally monotone set that is not of type (ANA).    
\par
There is one issue that we wish to mention briefly.  ``quasidensity'' \big(see \eqref{EE1}\big) does not require $E\dbs$ for its definition, and $E\dbs$ is not mentioned explicitly in the statements of \Thms{STDthm}, \ref{NIBRthm} and \ref{ANAthm}, but our proofs of all of these results use \Thm{CONJTOUCHthm}, which does depend on ($B^*$ hence) $E\dbs$, at one point or another.   This raises the question whether there are proofs of any of these results that do not depend on $E\dbs$.
\par
We now discuss the analysis in the first part of this paper, \Secs{SNsec}--\ref{PHIsec} and \ref{LINsec}, which provides the theoretical underpinnings for the results described above.   A glance at the condition for the ``quasidensity'' of subsets of $E \times E^*$ in \eqref{EE1} should convince the reader that the sheer length of the expression in this condition would make the concept hard to study.   In \Secs{SNsec}--\ref{PHIsec}, we show how to embed the analysis in a more general situation (``Banach SN spaces'') for which the notation is much more concise.   The definition of $r_L$--density in this more general situation can be found in \Def{RLdef}.
\par
Banach SN spaces are defined in \Def{SNdef}.   Banach SN spaces possess a quadratic form (denoted by $q_L$), and much of the analysis in \Secs{SNsec}--\ref{RLsec} is devoted to a study of those proper convex functions that dominate this quadratic form \big(denoted by $\PCQ(B)$\big).   If $f$ is such a function, the equality set is denoted by $\big\{B| f = q_L\big\}$.   The nonnegative function $r_L$ is  defined to be $\half\|\cdot\|^2 + q_L$.
\par
``$L$--positive sets'' (which generalize monotone subsets of $E \times E^*$) are defined in \Sec{LPOSsec}.   In \Sec{RLsec}, we introduce the concept of ``$r_L$--density''.   The first main result here is \Thm{RLEQthm}, in which we give two conditions equivalent to the $r_L$--density of a set of the form $\big\{B| f = q_L\big\}$.   The proof of the nontrivial part of \Thm{RLEQthm} is motivated by  Voisei--Z\u{a}linescu, \cite[Theorem~2.12, p.\ 1018]{VZ}.  Our analysis goes by way of the concept of ``touching function'', defined in \Def{TOUCHdef}.   This concept was used in \cite{VZ} in the $E \times E^*$ case, though unnamed.   The second main result in \Sec{RLsec} is the ``theorem of the touching conjugate", \Thm{Kthm}.  
\par
The main result in \Sec{SLsec} is \Thm{CONJTOUCHthm}, in which we give a characterization of the touchingness of a function in terms of its conjugate and the function $s_L$ defined on $B^*$ in \Def{SLdef}.   The rather arcane definition of $s_L$ is obtained by working backwards from \Thm{CONJTOUCHthm}(a), but it reduces to the simple form exhibited in \Lem{EESLlem} in the $E \times E^*$ case.  
\par
If $A$ is an $L$--positive subset of a Banach SN space, we define in \Sec{PHIsec} functions $\Phi_A$ and $\Theta_A$, which will be used extensively in what follows.   $\Phi_A$ is a generalization to Banach SN spaces of the ``Fitzpatrick function'' of a monotone set, which was originally introduced in \cite{FITZ}, but lay dormant until it was rediscovered by  Mart\'\i nez-Legaz and Th\'era in \cite{MLT}.   If $A$ is maximally $L$--positive, we give conditions in \Def{MARKdef} for an element $g$ of $\PC(B^*)$ to be a ``marker function'' for $A$, and we show in \Thm{MARKthm} how marker functions can be used to characterize the $r_L$--density of $A$.   Marker functions have implications for the Fitzpatrick extension, which are discussed in \Sec{FITZCHARsec}. 
\par
\Sec{LINsec} is about a closed linear $L$--positive subspace, A, of a Banach SN space, $B$, and its polar subspace, $A^0 \subset B^*$.    The main result here, in\break \Thm{LINthm}, is that $A$ is $r_L$--dense if, and only if, $\sup s_L(A^0) \le 0$.   This is the specific result (already alluded to) that depends on \Thm{RTRthm} for its proof, and is applied to monotone relations in \Sec{LINMONsec}.
\par
In \Sec{ZAGsec}, we show how Banach SN spaces lead to a generalization of a inequality due to Zagrodny, which was used to prove that the closure of the domain and the range of maximally monotone operator of type (NI) is convex.     It was worthy of note that Zagrodny established these results before the approach via ``type (ED)'' was known.   \big(As we have already mentioned, the corresponding results appear in this paper in \Thm{MFBRthm}(d).\big)
\par
The author would like to thank Mircea Voisei, Constantine Z\u{a}linescu, Maicon Marques Alves and Benar Svaiter, whose birthday presents \cite{VZ} and \cite{ASNI} had a considerable influence on the writing of this paper.
\section{SN maps and Banach SN spaces}\label{SNsec}
We start off by introducing some Banach space notation.
\begin{definition}\label{PCdef} 
If $X$ is a nonzero real Banach space and $f\colon\ X \to \rbar$, we write $\dom\,f$ for the set $\big\{x \in X\colon\ f(x) \in \RR\big\}$.   $\dom\,f$ is the {\em effective domain} of $f$.   We say that $f$ is {\em proper} if $\dom\,f \ne \emptyset$.   We write $\PC(X)$ for the set of all proper convex functions from $X$ into $\rbar$ and $\PCLSC(X)$ for the set of all proper convex lower semicontinuous functions from $X$ into $\rbar$.   We write $X^*$ for the dual space of $X$ \big(with the pairing $\bra\cdot\cdot\colon X \times X^* \to \RR$\big).  If $f \in \PCLSC(X)$ then, as usual, we define the {\em Fenchel conjugate}, $f^*$, of $f$ to be the function on $X^*$ given by
\begin{equation}\label{FSTAR}
x^* \mapsto \supn_X\big[x^* - f\big]\qquad(x^* \in X^*).
\end{equation}
If $g \in \PC(X^*)$ then we define the {\em Fenchel preconjugate}, $^*g$, of $g$ to be the function on $X$ given by
\begin{equation}\label{GPRESTAR}
x \mapsto \supn_{x^* \in X^*}\big[\bra{x}{x^*} - g(x^*)\big]\qquad(x \in X).
\end{equation}
We write $X\dbs$ for the bidual of $X$ \big(with the pairing $\bra\cdot\cdot\colon X^* \times X\dbs \to \RR$\big).   If $f \in \PCLSC(X)$ and $f^* \in \PCLSC(X^*)$, we define $f\dbs\colon X\dbs \to \rbar$ by $f\dbs(x\dbs) := \sup_{X^*}\big[x\dbs - f^*\big]$.   If $x \in X$, we write $\wh x$ for the canonical image of $x$ in $X\dbs$, that is to say\quad $(x,x^*) \in X \times X^* \lr \bra{x^*}{\wh x} = \bra{x}{x^*}$.\quad   We write $X_1$ for the closed unit ball of $X$.   If $Y \subset X$, we write $\I_Y$ for the {\em indicator function of $Y$}, defined by $\I_Y(x) = 0$ if $x \in Y$ and $\I_Y(x) = \infty$ if $x \in X \setminus Y$.\break   If $f,g\colon X \to \,[{-}\infty,\infty]$, then we write $\big\{X|f = g\big\}$ and $\big\{X|f \le g\big\}$ for the sets $\big\{x \in X,\ f(x) = g(x)\big\}$ and $\big\{x \in X,\ f(x) \le g(x)\big\}$, respectively.
\smallbreak
If $E$ and $F$ are nonzero Banach spaces then we define the projection maps $\pi_1$ and $\pi_2$ by $\pi_1(x,y) := x$ and $\pi_2(x,y) := y$ \big($(x,y) \in E \times F$\big).
\end{definition}
We will use the following result in \Thm{RLEQthm}:

\begin{lemma}[A boundedness result]\label{EXNlem}
Let $X$ be a nonzero real Banach space and $f \in \PC(X)$.   Suppose, further, that $m := \infn_{x \in X}\big[f(x) + \half\|x\|^2\big] \in \RR$, $y,z \in X$, $f(y) + \half\|y\|^2 \le m + 1$ and $f(z) + \half\|z\|^2 \le m + 1$.   Then $\|y\| \le \|z\| + 3$.
\end{lemma}
\begin{proof}
We have $m \le f\big(\half y + \half z\big) + \half\|\half y + \half z\|^2 \le \half f(y) + \half f(z) + \eighth\big[\|y\| + \|z\|\big]^2$.   Thus $m + \eighth\big[\|y\| - \|z\|\big]^2 \le \half f(y) + \half f(z) + \fourth\|y\|^2 + \fourth\|z\|^2$.   Consequently,\break $m + \eighth\big[\|y\| - \|z\|\big]^2 \le \half\big(f(y) + \half\|y\|^2\big) + \half\big(f(z) + \half\|z\|^2\big) \le \half(m + 1) + \half(m + 1)$.   Thus $\big[\|y\| - \|z\|\big]^2 \le 8$, which gives the required result.
\end{proof}
We now introduce {\em SN maps} and {\em Banach SN spaces} \big(which were called\break {\em Banach SNL spaces} in \cite{LINPOS}\big).
\begin{definition}\label{SNdef}
Let $B$ be a nonzero real Banach space.   A {\em SN map on $B$} (``SN'' stands for ``symmetric nonexpansive''), is a linear map $L\colon\ B \to B^*$ such that
\begin{equation}\label{SYM}
\|L\| \le 1\quand\all\ b,c\in B, \bra{b}{Lc} = \bra{c}{Lb}.
\end{equation}
A {\em Banach SN space} $(B,L)$ is a nonzero real Banach space $B$ together with a SN map $L\colon\ B \to B^*$.  From now on, we suppose that $(B,L)$ is a Banach SN space.   We define the even functions $q_L$ and $r_L$ on $B$ by\quad $q_L(b) := \half\bra{b}{Lb}$ (``$q$'' stands for ``quadratic'') and $r_L := \half\|\cdot\|^2 + q_L$.   Since $\|L\| \le 1$, for all $b \in B$, $|q_L(b)| = \half|\bra{b}{Lb}| \le \half\|b\|\|Lb\| \le \half\|b\|^2$, so that
\begin{equation}\label{RLPOS}
0 \le r_L \le \|\cdot\|^2\ \on\ B.
\end{equation}
For all $b,d \in B$, $|\half\|b\|^2 - \half\|d\|^2| = \half\big|\|b\| - \|d\|\big|\big(\|b\| + \|d\|\big) \le \half\|b - d\|\big(\|b\| + \|d\|\big)$ and, from \eqref{SYM},
$|q_L(b) - q_L(d)| = \half|\bra{b}{Lb} - \bra{d}{Ld}| =\half\big|\bra{b - d}{L(b + d)}\big| \le \half\|b - d\|\|b + d\|$.  Consequently,\quad $|r_L(b) - r_L(d)| \le \|b - d\|\big(\|b\| + \|d\|\big)$,\quad thus
\begin{equation}\label{XOTAfive}
r_L(b) \le \|b - d\|\big(\|b\| + \|d\|\big) + r_L(d)\quand r_L\hbox{ is continuous.}
\end{equation}
\end{definition}
\begin{notation}\label{Qnot}
We write
$$\PCQ(B) := \big\{f \in \PC(B)\colon\ f \ge q_L\ \on\ B\big\}$$
and
$$\PCLSCQ(B) := \big\{f \in \PCLSC(B)\colon\ f \ge q_L\ \on\ B\big\}.$$
\end{notation}
\Lem{ROOTlem} below will be used in \Lem{Llem}(a) and \Thm{RLEQthm}.
\begin{lemma}\label{ROOTlem}
Let $(B,L)$ be a Banach SN space, $f \in \PCQ(B)$ and $a,c \in B$.   Then
$$-q_L(a - c) \le 2(f - q_L)(a) + 2(f - q_L)(c).$$
\end{lemma}
\begin{proof}
$-q_L(a - c) = q_L(a + c) - 2q_L(a) - 2q_L(c) = 4q_L(\half a + \half c) - 2q_L(a) - 2q_L(c) \le 4f(\half a + \half c) - 2q_L(a) - 2q_L(c) \le 2f(a) + 2f(c) - 2q_L(a) - 2q_L(c)$.
\end{proof}
\begin{remark}\label{ROOTrem}
The following result stronger than \Lem{ROOTlem} was proved in\break \cite[Lemma~2.6, p.\ 231]{SSDMON}: if $f \in \PCQ(B)$ and $a,c \in B$ then
$$-q_L(a - c) \le \Big[\sqrt{(f - q_L)(a)} + \sqrt{(f - q_L)(c)}\Big]^2.$$
\end{remark}
If $B$ is any Banach space then $(B,0)$ is obviously a Banach SN space, $q_0 = 0$ and $r_0 = \half\|\cdot\|^2$.  There are many more interesting examples of Banach SN spaces.   The following are extensions of the examples in \cite[Examples~2.3 and 2.5, pp.\ 230--231]{SSDMON}.   More examples can be derived from \cite[Remark~6.7, p.\ 246]{SSDMON} and \cite{GMS}.   The significant example which leads to results on monotonicity appeared in \cite[Example~6.5, p.\ 245]{SSDMON}. We will return to it in Example~\ref{EEex} of this paper.   We note that some of the above examples were expressed in term of the bilinear form $\xbra{\cdot}{\cdot}\colon\ (b,c) \mapsto \bra{b}{Lc}$ rather than the map $L$.
\begin{example}\label{Hex}
Let $B$ be a Hilbert space with inner product $(b,c) \mapsto \bra{b}{c}$ and $L\colon B \to B$ be a nonexpansive self--adjoint linear operator.   Then $(B,L)$ is a Banach SN space.   Here are three special cases of this example:
\smallbreak
(a)\enspace $\lambda \in\,]0,1\,]$ and, for all $b \in B$, $Lb = \lambda b$.   Here $r_L(b) = \half(1 + \lambda)\|b\|^2$.
\smallbreak
(b)\enspace $\lambda \in \,]0,1\,]$ and, for all $b \in B$, $Lb = - \lambda b$.   Here $r_L(b) = \half(1 - \lambda)\|b\|^2$.
\smallbreak
(c)\enspace $\lambda \in \,]0,1\,]$, $B = \RR^3$ and $L(b_1,b_2,b_3) = \lambda(b_2,b_1,b_3)$.   Here
$$r_L(b_1,b_2,b_3) = \half\big(b_1^2 + 2\lambda b_1b_2 + b_2^2 + (1 + \lambda)b_3^2\big).$$
\end{example}
\section{$L$--positive sets}\label{LPOSsec}
Let $A \subset B$.   We say that $A$ is {\em $L$--positive} \big(\cite[Section~2, pp.\ 604--606]{LINPOS}\big) if $A \ne \emptyset$ and\quad $a,c \in A \lr q_L(a - c) \ge 0$.\quad In Example~\ref{Hex}(a), all nonempty subsets of $B$ are $L$--positive and, in Example~\ref{Hex}(b), the only $L$--positive subsets of $B$ are the singletons. In Example~\ref{Hex}(c) when $\lambda = 1$, the $L$--positive sets are explored in \cite[Example~3.2(c), p.\ 262]{POLAR}, \cite[Example~2.3(c), p. 606]{LINPOS} (and other places).
\begin{definition}\label{FATdef}
Let $(B,L)$ be a Banach SN space and $f \in \PC(B)$.   We define the function $f^@$ on $B$ by
\begin{equation}\label{FAT}
f^@(b) := f^*(Lb) = \supn_B\big[Lb - f\big]\qquad(b \in B).
\end{equation}
\end{definition}
\Lem{Llem} contains three fundamental properties of Banach SN spaces, and will be used in \Thm{RLEQthm}, \Thm{Kthm}, \Lem{PHIAlem}, \Lem{IFMARKlem},\break \Thm{INVARthm} and \eqref{AFMON}.   \Lem{Llem}(a) is suggested by Burachik--Svaiter,\break \cite[Theorem~3.1, pp. 2381--2382]{BS} and Penot, \cite[Proposition 4\big((h)$\lr$(a)\big), pp. 860--861]{PENOT}, and is equivalent to \cite[Lemma~2.9, p.\ 232]{SSDMON}.   \Lem{Llem}(b,c) are equivalent to \cite[Lemma~2.12(a,b), p.\ 233]{SSDMON}.
\begin{lemma}\label{Llem}
Let $(B,L)$ be a Banach SN space and $f \in \PCQ(B)$.   Then:
\smallbreak
\noindent
{\rm(a)}\enspace If $\big\{B|f = q_L\big\} \ne \emptyset$ then $\big\{B|f = q_L\big\}$ is an $L$--positive subset of $B$.
\smallbreak
\noindent
{\rm(b)}\enspace Let $a,b \in B$ and $f(a) = q_L(a)$.   Then $q_L(a) \ge \bra{b}{La} - f(b)$.
\smallbreak
\noindent
{\rm(c)}\enspace $\big\{B|f = q_L\big\} \subset \big\{B|f^@ = q_L\big\}$.
\end{lemma}
\begin{proof}
(a)\enspace This is immediate from \Lem{ROOTlem}.  As for (b), let $\lambda \in \,]0,1[\,$.   Then
\begin{align*}
\lambda f(b) + (1 - \lambda)q_L(a)
&= \lambda f(b) + (1 - \lambda)f(a)\\
&\ge f\big(\lambda b + (1 - \lambda)a\big) \ge q_L\big(\lambda b + (1 - \lambda)a\big)\\
&= \lambda^2q_L(b) + \lambda(1 - \lambda)\bra{b}{La} + (1 - \lambda)^2q_L(a).
\end{align*}
Thus\quad $\lambda f(b) + \lambda(1 - \lambda)q_L(a) \ge \lambda^2q_L(b) + \lambda(1 - \lambda)\bra{b}{La}$,\quad and (b) follows by dividing by $\lambda$, letting $\lambda \to 0$ and rearranging the terms.
\smallbreak
Now let $a \in B$ and $f(a) = q_L(a)$.   Taking the supremum over $b$ in (b) and using \eqref{FAT}, we see that $q_L(a) \ge f^@(a)$.   On the other hand, we also have $f^@(a) \ge \bra{a}{La} - f(a) = 2q_L(a) - q_L(a) = q_L(a)$.  Thus $f^@(a) = q_L(a)$.   This completes the proof of (c).
\end{proof}
\section{$r_L$--dense sets and touching functions}\label{RLsec}
\begin{definition}\label{RLdef}
Let $A$ be a subset of a Banach SN space $(B,L)$.   We say that $A$ is {\em $r_L$--dense in} $B$ if, for all $c \in B$, $\inf r_L(A - c) \le 0$.
\end{definition}
\noindent
If $B$ is any Banach space, $r_0$--density is clearly identical to norm--density.   The same is true for Example~\ref{Hex}(a) for all $\lambda \in\,]0,1\,]$ and Example~\ref{Hex}(b) for all $\lambda \in\,]0,1[\,$.   In Example~\ref{Hex}(b) when $\lambda = 1$, every nonempty subset of $B$ is $r_L$--dense in $B$.
\smallbreak
We will also consider the following strengthening of the condition of $r_L$--density: we will say that $A$ is {\em stably $r_L$--dense in} $B$ if, for all $c \in B$, there exists $K_c \ge 0$ such that
\begin{equation}
\inf\big\{r_L(a - c)\colon\ a \in A,\ \|a - c\| \le K_c\big\} \le 0.
\end{equation}
The concept of stable $r_L$--density will be used in the proof of \Thm{MFBRthm}(a).
\begin{definition}\label{TOUCHdef}
Let $(B,L)$ be a Banach SN space, $f \in \PCQ(B)$ and $c \in B$.   \eqref{RLPOS} implies that $\infn_{d \in B}\big[(f - q_L)(d) + r_L(d - c)\big] \ge 0$.    We say that $f$ is {\em touching} if
\begin{equation}\label{TOUCH2}
f \in \PCQ(B)\hbox{ and, }\all\ c \in B,\ \infn_{d \in B}\big[(f - q_L)(d) + r_L(d - c)\big] \le 0.
\end{equation}
\end{definition}
\begin{lemma}[Lower semicontinuous envelope]\label{HKlem}
Let $(B,L)$ be a Banach SN space, $h \in \PCQ(B)$ and $\hunder$ be the lower semicontinuous envelope of $h$.   Then:
\smallbreak\noindent
{\rm(a)}\enspace $\hunder \in \PCLSCQ(B)$.
\smallbreak\noindent
{\rm(b)}\enspace  Let $c \in B$.   Then we have
\begin{equation}\label{HEQK}
\infn_{d \in B}\big[(\hunder - q_L)(d) + r_L(d - c)\big] = \infn_{d \in B}\big[(h - q_L)(d) + r_L(d - c)\big].
\end{equation}
\smallbreak\noindent
{\rm(c)}\enspace $\hunder$ is touching if, and only if, $h$ is touching.
\smallbreak\noindent
{\rm(d)}\enspace $\hunder^@ = h^@$ on $B$.   
\end{lemma}
\begin{proof}
\underbar{$h$} is the (convex) function whose epigraph is the closure of the epigraph of $h$.   It is well known that \hunder\ is also the largest lower semicontinuous function on $B$ such that $\hunder \le h$ on $B$.   It is also well known that $\hunder^* = h^*$ on $B^*$.
\smallbreak
(a)\enspace Since $h \in \PCQ(B)$, $q_L \le h$ on $B$ thus, since $q_L$ is (continuous hence) lower semicontinuous on $B$, $q_L \le \hunder$ on $B$, from which $\hunder \in \PCLSCQ(B)$.
\smallbreak
(b)\enspace Since $\hunder \le h$ on $B$, the inequality ``$\le$'' in \eqref{HEQK} is obvious.   As we observed in \Def{TOUCHdef}, we have $\infn_{d \in B}\big[(h - q_L)(d) + r_L(d - c)] \ge 0$.   Now let $m := \infn_{d \in B}\big[(h - q_L)(d) + r_L(d - c)]$, so that $m \in \RR$ and, for all $d \in B$, $h(d) \ge q_L(d) - r_L(d - c) + m$.   The function $q_L - r_L(\cdot - c) + m$ is  (continuous hence) lower semicontinuous on $B$ and so, for all $d \in B$, $\hunder(d) \ge q_L(d) - r_L(d - c) + m$, that is to say, $(\hunder - q_L)(d) + r_L(d - c) \ge m$, which gives the inequality ``$\ge$'' in \eqref{HEQK}.
\smallbreak
(c) is immediate from (a), (b) and \eqref{TOUCH2}.\smallbreak
(d) is immediate since $\hunder^@ = \hunder^* \circ L = h^* \circ L =  h^@$ on $B$.
\end{proof}
\smallbreak
In the first main result of this section, \Thm{RLEQthm}, we give two characterizations of $r_L$--density for certain sets of the form $\big\{B|f = q_L\big\}$, including the unexpected result that, for these sets, $r_L$--density implies stable $r_L$--density.\break   \Thm{RLEQthm} and its consequence \Cor{RLEQcor} will be used in \Thm{Kthm}, \Thm{CONJTOUCHthm}, \Cor{PHICRITcor}, \Thm{MARKthm}, \Thm{RTRthm} and \Thm{LINthm}.
\begin{theorem}[The $r_L$--density of certain coincidence sets]\label{RLEQthm}
Let $(B,L)$ be a Banach SN space, $h \in \PCQ(B)$ and $\hunder$ be the lower semicontinuous envelope of $h$ {\em(since $q_L$ is continuous, $\big\{B|\hunder = q_L\big\}$ is closed)}.  Then the conditions {\rm(a)--(c)} are equivalent:
\smallbreak\noindent
{\rm(a)}\enspace $\big\{B|\hunder = q_L\big\}$ is an $r_L$--dense $L$--positive subset of $B$.
\smallbreak\noindent
{\rm(b)}\enspace $\hunder$ is touching or, equivalently {\em\big(from \Lem{HKlem}(c)\big),} $h$ is touching.
\smallbreak\noindent
{\rm(c)}\enspace $\big\{B|\hunder = q_L\big\}$ is a stably $r_L$--dense $L$--positive subset of $B$.
\end{theorem}
\begin{proof}
Let $A := \big\{B|\hunder = q_L\big\}$.   Then, for all $c \in B$,
\begin{align*}
\inf_{d \in B}\big[(\hunder - q_L)(d) + r_L(d - c)\big]
& \le \inf_{a \in A}\big[(\hunder - q_L)(a) + r_L(a - c)\big] = \inf_{a \in A}r_L(a - c).
\end{align*}
It follows easily from \Defs{RLdef} and \ref{TOUCHdef}, and \Lem{HKlem} that (a)$\lr$(b).
\smallbreak
Suppose now that (b) is satisfied and $c \in B$.   Replacing $d$ by $b + c$,
\begin{equation*}
0 = \inf_{d \in B}\big[(\hunder - q_L)(d) + r_L(d - c)\big] = \inf_{b \in B}\big[\hunder(b + c) - \bra{b}{Lc} - q_L(c) + \half\|b\|^2\big].
\end{equation*}
\Lem{EXNlem} provides $N_c \ge 0$ such that $\hunder(b + c) - \bra{b}{Lc} - q_L(c) + \half\|b\|^2 \le 1 \lr \|b\| \le N_c$.   Thus $(\hunder - q_L)(d) + r_L(d - c)  \le 1 \lr \|d - c\| \le N_c$.   Let $\delta \in \,]0,\half[\,$.   Let $c_0 := c$.   If $n \ge 1$ and $c_{n-1}$ is known then, from (b) and \eqref{TOUCH2} with $c$ replaced by $c_{n - 1}$, we can choose $c_n$ inductively so that,
\begin{equation}\label{RLEQ1}
(\hunder - q_L)(c_n) + r_L(c_n - c_{n - 1}) \le \delta^{2n}.
\end{equation}
Let $n \ge 1$.  From \Lem{HKlem}(a) and \eqref{RLPOS}, $(\hunder - q_L)(c_n) \ge 0$ and $r_L(c_n - c_{n - 1}) \ge 0$, and so \eqref{RLEQ1} implies that
\begin{equation}\label{RLEQ3}
(\hunder - q_L)(c_n) \le \delta^{2n}
\end{equation}
and
\begin{equation}\label{RLEQ4}
r_L(c_n - c_{n - 1}) \le \delta^{2n}.
\end{equation}
Putting $n = 1$ in \eqref{RLEQ1}, we have $(\hunder - q_L)(c_1) + r_L(c_1 - c) \le \delta^2 < 1$ and so, from the choice of $N_c$ and also setting $n = 1$ in \eqref{RLEQ4},
\begin{equation}\label{RLEQ2}
\|c_1 - c\| \le N_c \quand r_L(c_1 - c) \le \delta^2.
\end{equation}
From \eqref{RLEQ4}, \Lem{ROOTlem}, \eqref{RLEQ3}, and the fact that $\delta^{2n + 2} \le \fourth\delta^{2n}$,
\begin{align*}
\half\|c_{n+1} - c_n\|^2 &= - q_L(c_{n+1} - c_n) + r_L(c_{n+1} - c_n) \le - q_L(c_{n+1} - c_n) + \delta^{2n + 2}\\
&\le 2(\hunder - q_L)(c_{n+1}) + 2(\hunder - q_L)(c_n) + \delta^{2n + 2}\\
&\le 2\delta^{2n + 2} + 2\delta^{2n} + \delta^{2n + 2} \le 3\delta^{2n},
\end{align*}
from which\quad $\|c_{n+1} - c_n\| \le 3\delta^n$.\quad  Thus $\limn_{n \to \infty}c_n$ exists.   Let $a := \limn_{n \to \infty}c_n$.  From \eqref{RLEQ3} and the lower semicontinuity of $\hunder - q_L$,\quad $(\hunder - q_L)(a) \le 0$,\quad from which $a \in \big\{B|\hunder = q_L\big\}$.    Also,  $\|a - c_1\| \le  \sum_{n = 1}^\infty\|c_{n+1} - c_n\| \le 3\sum_{n = 1}^\infty\delta^n \le 6\delta$ and so, from \eqref{RLEQ2},\quad $\|a - c\| \le \|a - c_1\| + \|c_1 - c\| \le 6\delta + N_c \le N_c + 3$.\quad Then \eqref{XOTAfive} (with $b = a - c$ and $d = c_1 - c$) and \eqref{RLEQ2} give 
\[r_L(a - c) \le \|a - c_1\|\big(\|a - c\| + \|c_1 - c\|\big) + r_L(c_1 - c)\le 6\delta\big(N_c + 3 + N_c\big) + \delta^2.\]
Letting $\delta \to 0$,\quad $\inf\big\{r_L(a - c)\colon\ a \in \big\{B| \hunder = q_L\big\},\ \|a - c\| \le N_c + 3\big\} \le 0$.\quad   Thus
$\big\{B|\hunder = q_L\big\}$ is stably $r_L$--dense in $B$.   In particular, $\big\{B|\hunder = q_L\big\} \ne \emptyset$.   From \Lem{Llem}(a), this set is also $L$--positive.   Thus (c) holds.   Since it is obvious that (c)$\lr$(a), this completes the proof of the theorem.
\end{proof}
\begin{corollary}[The lower semicontinuous case]\label{RLEQcor}
Let $(B,L)$ be a Banach SN space and $k \in \PCLSCQ(B)$.  Then the conditions {\rm(a)--(c)} are equivalent:
\par\noindent
{\rm(a)}\enspace $\big\{B|k = q_L\big\}$ is an $r_L$--dense $L$--positive subset of $B$.
\par\noindent
{\rm(b)}\enspace $k$ is touching.
\par\noindent
{\rm(c)}\enspace $\big\{B|k = q_L\big\}$ is a stably $r_L$--dense $L$--positive subset of $B$.
\end{corollary}
\begin{proof}
This is immediate from \Thm{RLEQthm} since $\kunder = k$.
\end{proof}
\begin{definition}\label{MAXPOSdef}
Let $A$ be a nonempty subset of a Banach SN space $(B,L)$.   We say that $A$ is {\em maximally $L$--positive} if $A$ is $L$--positive and $A$ is not properly contained in any other $L$--positive set.
\end{definition}  
The simple result contained in \Lem{RLMAXlem} connects the concepts of\break maximal $L$--positivity and $r_L$--density.   The converse result is not true:  the graph of the tail operator mentioned in the introduction is a closed maximally $L$--positive linear subspace of $\ell_1 \times \ell_\infty$ that is not $r_L$--dense (see \Ex{TAILex}).  \Lem{RLMAXlem} will be used in \Thm{Kthm}(b), \Cor{AUTOcor}, \Thm{RLMAXthm}(a) and \Cor{LINcor}.
\begin{lemma}[$r_L$--density and maximal $L$--positivity]\label{RLMAXlem}
Let $(B,L)$ be a Banach SN space and $A$ be a closed, $r_L$--dense $L$--positive subset of $B$.   Then $A$ is maximally $L$--positive.
\end{lemma}     
\begin{proof}
Suppose that $b \in B$ and\quad $A \cup \{b\}$\quad is $L$--positive.   Let $\eps > 0$.   By\break hypothesis, there exists $a \in A$ such that\quad $\half\|a - b\|^2 + q_L(a - b) = r_L(a - b) < \eps$.\quad   Since\quad $A \cup \{b\}$\quad is $L$--positive,\quad $q_L(a - b) \ge 0$,\quad and so\quad $\half\|a - b\|^2 \le \eps$.\break   However, $A$ is closed.   Thus, letting $\eps \to 0$, $b \in A$.
\end{proof}
\par
We now come to the second main result in this section.   It will be used in \Lems{TOUCHDlem} and \ref{TOUCHRlem}.
\begin{theorem}[The theorem of the touching conjugate]\label{Kthm}
Let $(B,L)$ be a Banach SN space and $h \in \PCQ(B)$ be touching.   Then:
\smallbreak\noindent
{\rm(a)}\enspace $h^@ \ge q_L$ on $B$ and $\hunder^@ \ge q_L$ on $B$.
\smallbreak\noindent
{\rm(b)}\enspace $\big\{B|h^@ = q_L\big\} = \big\{B|\hunder^@ = q_L\big\} = \big\{B|\hunder = q_L\big\}$, and this set is nonempty,  closed, stably $r_L$--dense in $B$ and maximally $L$--positive.
\smallbreak\noindent
{\rm(c)}\enspace $h^@$ is touching.
\end{theorem}
\begin{proof}
Let $c \in B$.   Then, since $q_L \le r_L$ on $B$, for all $d \in B$,
\smallbreak
\centerline{$h(d) - \bra{d}{Lc} + q_L(c) = (h - q_L)(d) + q_L(d - c) \le (h - q_L)(d) + r_L(d - c).$} 
\smallbreak
\noindent
Thus, from \eqref{TOUCH2}, $\inf_{d \in B}\big[h(d) - \bra{d}{Lc} + q_L(c)\big]\le 0$. It follows from \eqref{FAT} that $h^@(c) = \sup_{d \in B}\big[\bra{d}{Lc} - h(d)\big] \ge q_L(c)$.   Thus $h^@ \ge q_L$ on $B$, and (a) now follows since \Lem{HKlem}(d) implies that $\hunder^@ = h^@$ on $B$.
\smallbreak
From \Lem{HKlem}(a), \Lem{Llem}(c) \big(with $f := \hunder^@$\big), \Thm{RLEQthm}\big((b)$\lr$(c)\big) and \Lem{RLMAXlem}, $\hunder \in \PCLSCQ(B)$,  $\big\{B|\hunder^@ = q_L\big\} \supset \big\{B|\hunder = q_L\big\}$ and $\big\{B|\hunder = q_L\big\}$ is nonempty, closed, stably $r_L$--dense in $B$ and maximally $L$--positive.   From (a), $\hunder^@ \ge q_L$ on $B$ and \Lem{Llem}(a) \big(with $f := \hunder^@$\big) implies that $\big\{B|\hunder^@ = q_L\big\}$ is $L$--positive.   Thus \Lem{HKlem}(d) and the maximality of $\big\{B|\hunder = q_L\big\}$ give (b).
\smallbreak
(a) and (b) give $h^@ \ge q_L$ on $B$ and $\big\{B|h^@ = q_L\big\} \ne \emptyset$, thus we have $h^@ \in \PCLSCQ(B)$.   (c) follows from (b) and \Cor{RLEQcor}\big((a)$\lr$(b)\big), with $k := h^@$.
\end{proof}
\section{A dual characterization of touching}\label{SLsec}
\Thm{CONJTOUCHthm}, one of the central result of this paper, will be used in \Cor{PHICRITcor}, \Thm{MARKthm}, \Thm{RTRthm}, \Lem{TOUCHDlem}, \Lem{TOUCHRlem} and \Thm{LINthm}.   We start by defining a function $s_L$ on the dual space, $B^*$, of $B$ that plays a similar role to the function $q_L$ that we have already defined on $B$.  The definition of $s_L$ is anything but intuitive --- it was obtained by working backwards from \Thm{CONJTOUCHthm}.   In this connection, the formula obtained in \Lem{EESLlem} is very gratifying, and it shows that \Thm{CONJTOUCHthm}($\rl$) extends \cite[Remark~2.3]{VZ} and part of \cite[Theorem~4.2]{ASBR}, and \Thm{CONJTOUCHthm}($\lr$) extends \cite[Theorem~2.12]{VZ}.
\begin{definition}\label{SLdef}
Let $(B,L)$ be a Banach SN space.   We define the function $s_L\colon\ B^* \to \rbar$ by
\begin{equation}\label{SL1}
s_L(b^*) = \supn_{c \in B}\big[\bra{c}{b^*} - q_L(c) -\half\|Lc - b^*\|^2\big].
\end{equation}
$s_L$ is {\em quadratic} in the sense that $s_L(\lambda b^*) = \lambda^2 s_L(b^*)$ whenever $b^* \in B^*$ and $\lambda \in \RR \setminus\{0\}$.
Clearly, $s_0(b^*) = \supn_{c \in B}\big[\bra{c}{b^*} -\half\|b^*\|^2\big]$, from which $s_0(0) = 0$ and, if $b^* \in B^* \setminus \big\{0\big\}$, then $s_0(b^*) = \infty$.   In Example~\ref{Hex}(a), using the properties of a Hilbert space, for all $b^* \in B^* = B$ and $c \in B$,\quad $\bra{c}{b^*} - q_L(c) -\half\|Lc - b^*\|^2 = \half\|b^*\|^2/\lambda - \half(1 + \lambda)\|\lambda c - b^*\|^2/\lambda$,\quad and so \eqref{SL1} implies that $s_L(b^*) = \half\|b^*\|^2/\lambda$.\end{definition}
We recall that {\em touching} was defined in \eqref{TOUCH2}.    
\begin{theorem}\label{CONJTOUCHthm}
Let $(B,L)$ be a Banach SN space and $h \in \PCQ(B)$.   Then
$$h\ \hbox{is touching }\iff h^* \ge s_L\ \on\ B^*.$$
\end{theorem}
\begin{proof}
In what follows, for all $c \in B$, we write $h_c(b) := h(b + c) - \bra{b}{Lc} - q_L(c)$.   Following the analysis in \Thm{RLEQthm}, $h$ is touching if, and only if, for all $c \in B$, $\inf_{b \in B}\big[h_c(b) + \half\|b\|^2\big] \le 0$.   From Rockafellar's version of the Fenchel duality theorem \big(see, for instance, Rockafellar, \cite[Theorem~3(a), p.\ 85]{FENCHEL}, Z\u alinescu, \cite[Theorem~2.8.7(iii), p.\ 127]{ZBOOK}, or \cite[Corollary~10.3, p.\ 52]{HBM}\big), this is, in turn, equivalent to the statement that,  for all $c \in B$, $\big[{h_c}^*(b^*) + \half\|b^*\|^2\big] \ge 0$.   But, by direct computation, ${h_c}^*(b^*) = h^*(b^* + Lc) - \bra{c}{b^*} - q_L(c)$.   Thus $h$ is touching exactly when, for all $c \in B$, $\infn_{b^* \in B^*}\big[h^*(b^* + Lc) - \bra{c}{b^*} - q_L(c) + \half\|b^*\|^2\big] \ge 0$.   From  the substitution $b^* = d^* - Lc$, this is equivalent to the statement that, for all $c \in B$, $\infn_{d^* \in B^*}\big[h^*(d^*) - \bra{c}{d^* - Lc} - q_L(c) + \half\|d^* - Lc\|^2\big] \ge 0$.\break   It now follows from \eqref{SL1} that this is equivalent to the statement that $h^* \ge s_L$ on $B^*$.
\end{proof}
\section{$\Phi_A$ and $\Theta_A$ and marker functions}\label{PHIsec}
Throughout this section, $(B,L)$ will be a Banach SN space and $A$ will be an $L$--positive subset of $B$.   Some of the results of this section appear in greater generality in \cite{HBM}: here we discuss only what we will need in this paper.    
\begin{definition}[The definition of $\Phi_A$]\label{PHIAdef}
We define $\Phi_A\colon\ B \to \rbar$ by
\begin{align}
\all\ b \in B,\quad\Phi_A(b) &= \supn_A\big[Lb - q_L\big] := \supn_{a \in A}\big[\bra{a}{Lb} - q_L(a)\big]\label{PHI1}\\
&= q_L(b) - \inf q_L(A - b).\label{PHI6}
\end{align}
$\Phi_A$ is clearly lower semicontinuous.   If $b \in A$ then, since $A$ is  $L$--positive, $\inf q_L(A - b) = 0$, and so \eqref{PHI6} gives $\Phi_A(b) = q_L(b)$.   Thus
\begin{equation}\label{PHI7}
A \subset \big\{B|\Phi_A = q_L\big\}.     
\end{equation}  
\end{definition}
\begin{definition}[The definition of $\Theta_A$]\label{THdef}
We define $\Theta_A\colon\ B^* \to \rbar$ by
\begin{equation}\label{TH3}
\Theta_A(b^*) := \supn_{a \in A}\big[\bra{a}{b^*} - q_L(a)\big] = \supn_A\big[b^* - q_L\big]\quad(b^* \in B^*).
\end{equation}
\end{definition}
\begin{lemma}[Various properties of $\Phi_A$ and $\Theta_A$]\label{PHIAlem}
Let $A$ be maximally $L$--positive.   Then:
\begin{gather}
\Theta_A \circ L = \Phi_A\ \on\ B.\label{TH1}\\
\Phi_A \in \PCLSCQ(B)\quand \big\{B|\Phi_A = q_L\big\} = A.\label{PHI2}\\
{\Phi_A}^* \ge \Theta_A\ \on\ B^*.\label{TH2}\\
{\Phi_A}^@ \ge \Phi_A\ \on\ B.\label{PHI3}\\
{\Phi_A}^@ \in \PCLSCQ(B)\quand \big\{B|{\Phi_A}^@ = q_L\big\} = A.\label{PHI4}
\end{gather}
\end{lemma}
\begin{proof}
From \eqref{TH3} and \eqref{PHI1}, for all $b \in B$, $\Theta_A(Lb) := \supn_{a \in A}\big[\bra{a}{Lb} - q_L(a)\big] = \Phi_A(b)$.  This gives \eqref{TH1}. 
\smallbreak
If $b \in B$ and $\Phi_A(b) \le q_L(b)$ then \eqref{PHI6} gives $\inf q_L(A - b) \ge 0$.   From the maximality, $b \in A$ and so, from \eqref{PHI7}, $\Phi_A(b) = q_L(b)$.  Thus we have proved that $\Phi_A \ge q_L$ on $B$ and $\big\{B|\Phi_A = q_L\big\} \subset A$, and \eqref{PHI2} follows from \eqref{PHI7}.
\smallbreak
\eqref{FSTAR}, \eqref{PHI7} and \eqref{TH3} imply that, for all $b^* \in B^*$, ${\Phi_A}^*(b^*) = \sup_B\big[b^* - \Phi_A\big] \ge \sup_A\big[b^* - \Phi_A\big] = \sup_A\big[b^* - q_L\big] = \Theta_A(b^*)$.   This gives \eqref{TH2}. 
\smallbreak
\eqref{PHI3} is immediate from \eqref{TH2}, \eqref{FAT} and \eqref{TH1}.
\smallbreak
From \eqref{PHI2} and \Lem{Llem}(c), $\big\{B|{\Phi_A}^@ = q_L\big\} \supset \big\{B|\Phi_A = q_L\big\}$, and \eqref{PHI4} now follows from \eqref{PHI3}.
\end{proof}
\begin{corollary}\label{PHICRITcor}
Let $A$ be maximally $L$--positive (hence closed).    Then the\break conditions {\rm(a)--(d)} are equivalent:
\par\noindent
{\rm(a)}\enspace $A$ is an $r_L$--dense $L$--positive subset of $B$.
\par\noindent
{\rm(b)}\enspace $\Phi_A$ is touching.
\par\noindent
{\rm(c)}\enspace $A$ is a stably $r_L$--dense, $L$--positive subset of $B$.
\par\noindent
{\rm(d)}\enspace ${\Phi_A}^* \ge s_L$ on $B^*$.
\end{corollary}
\begin{proof}
Using \eqref{PHI2}, the equivalence of (a), (b) and (c) follows from \Cor{RLEQcor} with $k := \Phi_A$, and \Thm{CONJTOUCHthm} gives the equivalence with (d). 
\end{proof}
\begin{corollary}[Automatic stable $r_L$--density]\label{AUTOcor}
Every closed, $r_L$--dense $L$--positive subset of $B$ is stably $r_L$--dense.
\end{corollary}
\begin{proof}
This is immediate from \Lem{RLMAXlem} and \Cor{PHICRITcor}\big((a)$\lr$(c)\big).
\end{proof}
\begin{corollary}[Restricted converse to \Lem{RLMAXlem}]\label{CONVcor}
Let $L$ be an isometry of $B$ {\em onto} $B^*$ and $A$ be maximally $L$--positive.   Then $A$ is closed and stably $r_L$--dense in $B$.
\end{corollary}
\begin{proof}
Let $b^* \in B^*$.   Choose $b \in B$ such that $Lb = b^*$.  Then, from \eqref{SL1},
\begin{align*}
q_L(b) - s_L(b^*) &= \infn_{c \in B}\big[q_L(b) - \bra{c}{Lb} + q_L(c) +\half\|Lc - Lb\|^2\big]\\
&= \infn_{c \in B}\big[q_L(c - b) + \half\|c - b\|^2\big] = \infn_{c \in B}r_L(c - b) = 0.
\end{align*}
Consequently, $q_L(b) = s_L(b^*)$.   Thus, from \eqref{PHI4},\quad ${\Phi_A}^*(b^*) = {\Phi_A}^*(Lb) = {\Phi_A}^@(b) \ge q_L(b) = s_L(b^*)$.\quad It now follows from \Cor{PHICRITcor}\big((d)$\lr$(c)\big) that $A$ is stably $r_L$--dense in $B$, and the maximality implies that $A$ is closed.
\end{proof}
\begin{definition}[Marker functions]\label{MARKdef}
Let $A$ be maximally $L$--positive and $g \in \PC(B^*)$.   We say that $g$ is a {\em marker function} for $A$ if $g$ is $w(B^*,B)$--lower semicontinuous,
\begin{equation}\label{MARK3}
g \le {\Phi_A}^*\ \on\ B^*
\end{equation}   
and
\begin{equation}\label{GTH1}
g \ge \Theta_A\ \on\ B^*.
\end{equation}
It is clear that if $g_1$ and $g_2$ are marker functions for $A$, $\lambda_1,\lambda_2 > 0$ and $\lambda_1 + \lambda_2 = 1$ then $\lambda_1g_1 + \lambda_2g_2$ is a marker function for $A$. 
\end{definition}
\begin{lemma}[Two significant cases]\label{PHSTMAlem}
Let $A$ be maximally $L$--positive.   Then ${\Phi_A}^*$ and $\Theta_A$ are marker functions for $A$.
\end{lemma}
\begin{proof}
${\Phi_A}^*$ and $\Theta_A$ are obviously convex and $w(B^*,B)$--lower semicontinuous.
\par
First, let $g := {\Phi_A}^*$.   From \eqref{TH2},  $g$ satisfies \eqref{GTH1}, and it is clear that $g$ satisfies \eqref{MARK3}.
Thus ${\Phi_A}^*$ is a marker function for $A$.
\par
Next, let $g =: \Theta_A$.   It is clear that $g$ satisfies \eqref{GTH1} and, from \eqref{TH2}, $g$ satisfies \eqref{MARK3}.   Thus $\Theta_A$ is a marker function for $A$.
\end{proof}
\begin{lemma}\label{IFMARKlem}
Let $A$ be maximally $L$--positive and $g$ be a marker function for $A$.   Then $^*g \in \PCLSCQ(B)$ and $\big\{B|\ ^*g = q_L\big\} = A$.
\end{lemma}
\begin{proof}
From \eqref{MARK3} and the Fenchel--Moreau theorem \big(see Moreau, \cite[Sections 5--6, pp.\ 26--39]{MOREAU}\big), $^*g \ge ^*\hskip-4pt({\Phi_A}^*) = \Phi_A$ on $B$ and so, from \eqref{PHI2},
\begin{equation}\label{MARK1}
^*g \ge \Phi_A \ge q_L\ \on\ B.
\end{equation}
Consequently, $^*g \in \PCLSCQ(B)$.   From \eqref{GTH1} and \eqref{TH3}, for all $b^* \in B^*$ and $a \in A$,
$g(b^*) \ge \bra{a}{b^*} - q_L(a)$.   It follows that, for all $a \in A$ and $b^* \in B^*$, $\bra{a}{b^*} - g(b^*) \le q_L(a)$.   Thus, from    \eqref{GPRESTAR}, $^*g(a) \le q_L(a)$.   Consequently,
\begin{equation}\label{MARK2}
^*g \le q_L\ \on\ A.
\end{equation}
From \eqref{MARK1} and \eqref{MARK2}, $\big\{B|\ ^*g = q_L\big\} \supset A$.   From \Lem{Llem}(a), $\big\{B|\ ^*g = q_L\big\}$ is $L$--positive, and result follows from the maximality of $A$.
\end{proof}        
\Thm{MARKthm} below will be used in \Lem{STABlem} and \Thms{INVARthm} and \ref{AFMAXthm}.
\begin{theorem}[Marker function characterization of $r_L$--density]\label{MARKthm}
Let $A$ be maximally $L$--positive and $g$ be a marker function for $A$.   Then $A$ is $r_L$--dense in $B$ if, and only if, $g \ge s_L$ on $B^*$.
\par
In particular, $A$ is $r_L$--dense in $B$ if, and only if, $\Theta_A \ge s_L$ on $B^*$.\end{theorem}
\begin{proof}
From \Lem{IFMARKlem}, $^*g \in \PCLSCQ(B)$ and $\big\{B|\ ^*g = q_L\big\} = A$.   From \Cor{RLEQcor} with $k :=\ ^*g$, and \Thm{CONJTOUCHthm} with $h :=\ ^*g $, $A$ is $r_L$--dense in $B$ if, and only if, $(^*g)^* \ge s_L$ on $B^*$.   The result follows since the Fenchel--Moreau theorem for the convex lower semicontinuous function $g$ on the locally convex space $\big(B^*,w(B^*,B)\big)$ implies that $(^*g)^* = g$ on $B^*$.  
\end{proof}        
\section{$E \times E^*$}\label{EEsec}
We suppose for the rest of this paper that $E$ is a nonzero Banach space.\break   Example~\ref{EEex} below appeared in \cite[Example~3.1, pp.\ 606--607]{LINPOS}.
\begin{example}\label{EEex}
Let $B := E \times E^*$ and, for all $(x,x^*) \in B$, we define the norm on $B$ by $\|(x,x^*)\| := \sqrt{\|x\|^2 + \|x^*\|^2}$.   We represent $B^*$ by $E^* \times E\dbs$, under the pairing
$$\Bra{(x,x^*)}{(y^*,y\dbs)} := \bra{x}{y^*} + \bra{x^*}{y\dbs},$$
and define $L\colon\ B \to B^*$ by $L(x,x^*) := (x^*,\wh{x})$.   Then $(B,L)$ is a Banach SN space and, for all $(x,x^*) \in B$, we have $q_L(x,x^*) = \bra{x}{x^*}$ and $r_L(x,x^*) = \half\|x\|^2 + \half\|x^*\|^2 + \bra{x}{x^*}$.   If $A \subset B$ then we say that $A$ is {\em quasidense} (resp. {\em stably quasidense}) if $A$ is $r_L$--dense (resp. {\em stably $r_L$--dense}) in $E \times E^*$ with respect to this value of $r_L$. So $A$ is quasidense exactly when
\begin{equation}\label{EE1}
\left.\begin{aligned}
(x,x^*) &\in B\lr\\
&\infn_{(s,s^*) \in A}\big[\half\|s - x\|^2 + \half\|s^* - x^*\|^2 + \bra{s - x}{s^* - x^*}\big] \le 0.
\end{aligned}\right\}
\end{equation}
\end{example}
If $A \subset E \times E^*$ then $A$ is $L$--positive exactly when $A$ is a nonempty monotone subset of $E \times E^*$ in the usual sense, and $A$ is maximally $L$--positive exactly when $A$ is a maximally monotone subset of $E \times E^*$ in the usual sense.   Any finite dimensional Banach SN space of the form described here must have {\em even} dimension, and there are many Banach SN spaces  of finite odd dimension.   See \cite[Remark~6.7, p.\ 246]{SSDMON}.
\par
It is worth making a few comments about the function $r_L$ is this context.   It appears explicitly in the ``perfect square criterion for maximality'' in the reflexive case in \cite[Theorem~10.3,\ p.\ 36]{MANDM}.   It also appears explicitly (still in the reflexive case) in Simons--Z\u{a}linescu \cite{SZNZ}, with the symbol ``$\Delta$''.   It was used in the nonreflexive case by Zagrodny in \cite{ZAGRODNY} (see Remarks~\ref{NIrem} and \ref{ZAGrem}).  
\par
The dual norm on $B^*$ is given by  $\|(y^*,y\dbs)\| := \sqrt{\|y^*\|^2 + \|y\dbs\|^2}$.  We define $\LT\colon\ B^* \to B\dbs$ by $\LT(y^*,y\dbs) = \big(y\dbs,\wh{y^*}\big)$. Then $\big(B^*,\LT\big)$ is a Banach SN space and, for all $(y^*,y\dbs) \in B^*$, $\qLt(y^*,y\dbs) = \bra{y^*}{y\dbs}$.   The Banach SN spaces $(B,L)$ and $\big(B^*,\LT\big)$ are related by the following result (see \cite[eqn. (53), p.\ 245]{SSDMON}):
\begin{lemma}\label{EEEElem}
$L(B)$ is $\rLt$--dense in $B^*$.
\end{lemma} 
\begin{proof}
Let $(x^*,x\dbs) \in B^*$.   The definition of $\|x\dbs\|$ provides an element $z^*$ of $E^*$ such that $\|z^*\| \le \|x\dbs\|$ and $\bra{z^*}{x\dbs} \le -\|x\dbs\|^2 + \eps$, from which it follows that\quad $\rLt(z^*,x\dbs) = \bra{z^*}{x\dbs} + \half\|z^*\|^2 + \half\|x\dbs\|^2 \le \bra{z^*}{x\dbs} + \|x\dbs\|^2 \le \eps$. Thus
\begin{align*}
\rLt\big(L(0,x^* - z^*) - (x^*,x\dbs)\big) &= \rLt\big((x^* - z^*,0) - (x^*,x\dbs)\big)\\ 
&= \rLt(-z^*,-x\dbs) = \rLt(z^*,x\dbs) \le \eps.
\end{align*}
This gives the required result.   
\end{proof}
\bigbreak
The following result will be used many times:
\begin{lemma}\label{EESLlem}
Let $(x^*,x\dbs) \in E^* \times E\dbs$.   Then
\smallbreak
\centerline{$s_L(x^*,x\dbs) = \bra{x^*}{x\dbs} = \qLt(x^*,x\dbs)$.} 
\end{lemma}
\begin{proof}
By direct computation from \eqref{SL1}, and using \Lem{EEEElem},   
\begin{align*}
\bra{x^*}{x\dbs} &- s_L(x^*,x\dbs)\\
&= \infn_{(y,y^*) \in B}\big[\bra{y^* - x^*}{\wh y - x\dbs} + \half\|L(y,y^*) - (x^*,x\dbs)\|^2\big]\\
&= \infn_{(y,y^*) \in B}\rLt\big(L(y,y^*) - (x^*,x\dbs)\big) = 0.
\end{align*}
Thus $s_L(x^*,x\dbs) = \bra{x^*}{x\dbs}$, which gives the desired result.
\end{proof}
\begin{theorem}[Quasidensity and maximality]\label{RLMAXthm}
Let $A \subset E \times E^*$ be monotone.
\par\noindent
{\rm(a)}\enspace Let $A$ be closed and quasidense.   Then $A$ is maximally monotone. 
\par\noindent
{\rm(b)}\enspace Let $E$ be reflexive and $A$ be maximally monotone.   Then $A$ is closed and stably quasidense.
\end{theorem}
\begin{proof}
This is immediate from \Lem{RLMAXlem} and \Cor{CONVcor}.
\end{proof}
We note that there is a novel application of \Thm{RTRthm} below to {\em linear $L$--positive sets} in \Lem{LINlem}.
\begin{theorem}[A generalization of Rockafellar's theorem on subdifferentials]
\label{RTRthm}
Let $k \in \PCLSC(E)$.  Then $G(\partial k)$ is stably quasidense and  maximally monotone.
\end{theorem}
\begin{proof}
First, fix $x_0 \in \dom\,k$.   From the Fenchel--Moreau theorem, $k(x_0) = \supn_{x^* \in E^*}\big[\bra{x_0}{x^*} - k^*(x^*)\big]$.   It follows that there exists $x_0^* \in \dom\,k^*$.   Let $f(x,x^*) := k(x) + k^*(x^*)$.  \big(Cf. \cite[Remark 2.13, p.\ 1019]{VZ}.\big) Since $f(x_0,x_0^*)  \in \RR$, $f \in \PCLSC(E \times E^*)$, and the Fenchel--Young inequality implies that\break $f \in \PCLSCQ(E \times E^*)$.   Now, using the Fenchel--Young inequality again and \Lem{EESLlem}, for all $(y^*,y\dbs) \in E^* \times E\dbs$,
\begin{equation}\label{RTR1}
\left.\begin{aligned}
f^*(y^*,y\dbs) &= \supn_{x \in E,\ x^* \in E^*}\big[	\bra{x}{y^*} + \bra{x^*}{y\dbs} - k(x) - k^*(x^*)\big]\\
&= \supn_{x \in E}\big[\bra{x}{y^*} - k(x)\big] + \supn_{x^* \in E^*}\big[\bra{x^*}{y\dbs} - k^*(x^*)\big]\\
&= k^*(y^*) + k\dbs(y\dbs) \ge \bra{y^*}{y\dbs} = s_L(y^*,y\dbs), 
\end{aligned}
\right\}
\end{equation}
and so \Thm{CONJTOUCHthm} implies that $f$ is touching, and \Cor{RLEQcor}\big((b)$\lr$(c)\big) implies that $\big\{E \times E^*|f = q_L\big\}$ is stably quasidense and maximally monotone.   The result follows since $\big\{E \times E^*|f = q_L\big\} = G(\partial k)$.
\end{proof}
\Lem{STABlem} will be used in \Thm{MFBRthm} to simplify certain computations.   The result of \Lem{STABlem} is certainly one that one would expect to be true.   The proof that we give here is surprisingly sophisticated, relying as it does on \Thm{MARKthm}.   We do not know if there is a simple proof using $r_L$ directly. 
\begin{lemma}\label{STABlem}
Let $\alpha,\beta > 0$ and $\Delta\colon\ E \times E^* \to E \times E^*$ be the ``deformation'' defined by $\Delta(x,x^*) := (x/\alpha,x^*/\beta)$.   Let $A$ be a closed, monotone and quasidense subset of $E \times E^*$.   Then $\Delta(A)$ is closed, monotone and stably quasidense.
\end{lemma}
\begin{proof}
From \Thm{RLMAXthm}(a), $A$ is maximally monotone and so \Thm{MARKthm} and \Lem{EESLlem} imply that $\Theta_A \ge s_L = \qLt$ on $E^* \times E\dbs$ thus, from \eqref{TH3}, for all $(x^*,x\dbs) \in E^* \times E\dbs$,
\begin{equation*}
\begin{aligned}
\Theta_{\Delta(A)}(x^*,x\dbs)
&= \supn_{(s,s^*) \in A}\big[\bra{s/\alpha}{x^*} + \bra{s^*/\beta}{x\dbs} - \bra{s}{s^*}/\alpha\beta\big]\\
&= \supn_{(s,s^*) \in A}\big[[\bra{s}{\beta x^*} + \bra{s^*}{\alpha x\dbs} - \bra{s}{s^*}\big]/\alpha\beta\\
&= \Theta_A(\beta x^*,\alpha x\dbs)/\alpha\beta \ge \qLt(\beta x^*,\alpha x\dbs)/\alpha\beta = \qLt(x^*,x\dbs).
\end{aligned}
\end{equation*}
It is obvious that $\Delta(A)$ is maximally monotone, and so \Thm{MARKthm} and \Cor{AUTOcor} imply that $\Delta(A)$ is stably quasidense.
\end{proof}
In order to simplify some notation in the sequel, if $S\colon\ E \toto E^*$, we will say that $S$ is {\em closed} if its graph, $G(S)$, is closed in $E \times E^*$, and we will say that $S$ is \emph{quasidense} (resp. \emph{stably quasidense})  if $G(S)$ is quasidense (resp. stably quasidense) in $E \times E^*$.   If $S$ is nontrivial and monotone, we shall write $\varphi_S$ for $\Phi_{G(S)}$.   We will switch freely between discussing multifunctions from $E$ into $E^*$ and subsets of $E \times E^*$ in what follows, depending on the context.   We have:
\begin{lemma}\label{PHISlem}
Let $S\colon\ E \toto E^*$ be closed, monotone and quasidense.   Then:
\begin{gather}
\varphi_S \in \PCLSCQ(E \times E^*) \quand \big\{E \times E^*|\varphi_S = q_L]\big\} = G(S).\label{PHIS1}\\
\varphi_S\hbox{ is touching}.\label{PHIS4}\\
D(S) \subset \pi_1\dom\,\varphi_S \quand R(S) \subset \pi_2\dom\,\varphi_S.\label{PHIS2}\\
{\varphi_S}^@ \in \PCLSCQ(E \times E^*) \quand \big\{E \times E^*|{\varphi_S}^@ = q_L\big\} = G(S).\label{PHIS3}
\end{gather}
\end{lemma}
\begin{proof}
From \Thm{RLMAXthm}(a), $S$ is maximally monotone.  \eqref{PHIS1} follows from \eqref{PHI2}; \eqref{PHIS4} follows from \Cor{PHICRITcor}\big((a)$\lr$(b)\big);    \eqref{PHIS2} follows from \eqref{PHIS1}; \eqref{PHIS3} follows from \eqref{PHI4}.
\end{proof}
\begin{theorem}[The out--of--range criterion for quasidensity]\label{OORthm}
Let $S\colon\ E \toto E^*$ be maximally monotone.   Then $S$ is stably quasidense if, and only if,
\begin{equation}\label{OOR1}
(w^*,w\dbs) \in \big(E^* \setminus R(S)\big) \times E\dbs \qlr {\varphi_S}^*(w^*,w\dbs) \ge \bra{w^*}{w\dbs}.
\end{equation}
\end{theorem}
\begin{proof}
``Only if'' is immediate from \Cor{PHICRITcor}\big((c)$\lr$(d)\big) and \Lem{EESLlem}.   Now if $(w^*,w\dbs) \in R(S) \times E\dbs$ then we can choose $w \in S^{-1}w^*$.   From \eqref{PHI7},\break ${\Phi_{G(S)}}^*(w^*,w\dbs) = {\varphi_S}^*(w^*,w\dbs) \ge \bra{w}{w^*} + \bra{w^*}{w\dbs} - \varphi_S(w,w^*) = \bra{w^*}{w\dbs}$, and ``if'' follows from \eqref{OOR1}, \Cor{PHICRITcor}\big((d)$\lr$(c)\big) and \Lem{EESLlem}. 
\end{proof}
\begin{corollary}[A  sufficient condition for quasidensity]\label{SURJcor}
Let $S\colon\ E \toto E^*$ be maximally monotone and $R(S) = E^*$.   Then $S$ is stably  quasidense.
\end{corollary}
\begin{proof}
This is immediate from \Thm{OORthm}, since $(E^* \setminus R(S)\big) \times E\dbs = \emptyset$.
\end{proof}
\par
The result given in \Ex{TAILex} below will be extended in \Thm{HTthm}.
\begin{example}[The tail operator]\label{TAILex}
Let $E = \ell_1$, and define $T\colon\ \ell_1 \mapsto \ell_\infty = E^*$ by $(Tx)_n = \sum_{k \ge n} x_k$.   It is well known that $T$, being a monotone linear operator with full domain, is maximally monotone.   Let $e^* := (1,1,\dots) \in {\ell_1}^* = \ell_\infty$.   Let $x \in \ell_1$, and write $\sigma = \bra{x}{e^*} = \sum_{n \ge 1}x_n$.   Clearly, $\|x\| \ge \sigma$.   Since $Tx \in c_0$, we also have $\big\|Tx - e^*\big\| = \sup_n\big|(Tx)_n - 1\big| \ge \lim_n\big|(Tx)_n - 1\big| = 1$.   Thus
\begin{equation}\label{TAIL1}
\left.\begin{aligned}
\bra{x}{Tx} &= \ts\sum_{n \ge 1}x_n\sum_{k \ge n}x_k = \sum_{n \ge 1}x_n^2 + \sum_{n \ge 1}\sum_{k > n}x_nx_k\cr
&\ge \ts\half\sum_{n \ge 1}x_n^2 + \sum_{n \ge 1}\sum_{k > n}x_nx_k = \half\sigma^2.
\end{aligned}\right\}
\end{equation}
It follows that
\begin{align*}
r_L\big((x,Tx) &- (0,e^*)\big) = \half\|x\|^2 + \half\|Tx - e^*\|^2 + \bra{x}{Tx - e^*}\\
&\ge \half\sigma^2 + \half + \bra{x}{Tx} - \sigma
\ge \half\sigma^2 + \half + \half\sigma^2 - \sigma = \sigma^2 + \half - \sigma \ge \fourth.
\end{align*} 
Consequently, $T$ is not quasidense. 
\end{example}
%
\section{Two sum theorems and the Fitzpatrick\break extension}\label{SUMSsec}
Let $X$ and $Y$ be nonzero Banach spaces.   \Lem{SZlem} below first appeared in Simons--Z\u{a}linescu \cite[Section~4, pp.\ 8--10]{SZNZ}.   It was subsequently generalized in \cite[Theorem 9, p.\ 882]{QUAD}, \cite[Corollary 5.4, pp.\ 121--122]{AST} and \cite[Theorem 4.1, p.\ 6]{MAR}.
\begin{lemma}\label{SZlem}
Let $f,g \in \PCLSC(X \times Y)$.   For all $(x,y) \in X \times Y$, let
$$h(x,y) := \infn_{v \in Y}\big[f(x,y - v) + g(x,v)\big] > -\infty.$$
Suppose that
\begin{equation}\label{SZ1}
\ts\bigcupn_{\lambda > 0}\lambda\big[\pi_1\dom\,f - \pi_1\dom\,g\big]\ \hbox{is a closed linear subspace of}\ X.\end{equation}
Then $h \in \PC(X \times Y)$ and, for all $(x^*,y^*) \in X^* \times Y^*$,
\begin{equation}\label{SZ8}
h^*(x^*,y^*) = \minn_{u^* \in X^*}\big[f^*(x^* -  u^*,y^*) + g^*(u^*,y^*)\big].
\end{equation}
\end{lemma}
\begin{lemma}\label{TOUCHDlem}
Let $f,g \in \PCLSCQ(E \times E^*)$ be touching.   For all $(x,x^*) \in E \times E^*$, let
\begin{equation}\label{TOUCHD1}
h(x,x^*) := \infn_{\xi^* \in E^*}\big[f(x,x^* - \xi^*) + g(x,\xi^*)\big].
\end{equation}
Suppose that
\begin{equation}\label{TOUCHD2}
\ts\bigcupn_{\lambda > 0}\lambda\big[\pi_1\dom\,f - \pi_1\dom\,g\big]\ \hbox{is a closed linear subspace of}\ E.\end{equation}
Then
\begin{equation}\label{TOUCHD4}
h \in \PCQ(E \times E^*),
\end{equation}
\begin{equation}\label{TOUCHD7}
h\hbox{ is touching},
\end{equation}
for all $(x,x^*) \in E \times E^*$,
\begin{equation}\label{TOUCHD3}
h^@(x,x^*) = \minn_{u^* \in E^*}\big[{f}^@(x,x^* - u^*) + {g}^@(x,u^*)\big] \ge \bra{x}{x^*},
\end{equation}
and
\begin{equation}\label{TOUCHD5}
\big\{E \times E^*|h^@ = q_L\big\}\hbox{ is closed, monotone and stably quasidense.}
\end{equation}
\end{lemma}
\begin{proof}
\eqref{TOUCHD2} implies that $\pi_1\,\dom\,f \cap \pi_1\,\dom\,g \ne \emptyset$, and so there exists\break $(x_0,y_0^*,z_0^*) \in E \times E^* \times E^*$ such that $f(x_0,y_0^*) \in \RR$ and $g(x_0,z_0^*) \in \RR$.   By hypothesis, $f \ge q_L$ and $g \ge q_L$ on $E \times E^*$.  Then, for all $(x,x^*) \in E \times E^*$,\quad $h(x,x^*) \ge \infn_{\xi^* \in E^*}\big[\bra{x}{x^* - \xi^*} + \bra{x}{\xi^*}\big] = \bra{x}{x^*}$\quand $h(x_0,y_0^* + z_0^*) \le f(x_0,y_0^*) + g(x_0,z_0^*) < \infty$, consequently \eqref{TOUCHD4} is satisfied.  From \Thm{CONJTOUCHthm} and \Lem{EESLlem}, for all $(x^*,x\dbs) \in E^* \times E\dbs$, 
\smallbreak
\centerline{${f}^*(x^*,x\dbs) \ge \bra{x^*}{x\dbs}\quand {g}^*(x^*,x\dbs) \ge \bra{x^*}{x\dbs}$.}
\smallbreak
\noindent
Thus \Lem{SZlem} (with $X := E$ and $Y := E^*$) and \Lem{EESLlem} imply that, for all $(x^*,x\dbs) \in E^* \times E\dbs$,
\begin{equation}\label{TOUCHD6}
\left.
\begin{aligned}
h^*&(x^*,x\dbs) \\
&= \minn_{u^* \in E^*}\big[f^*(x^* -  u^*,x\dbs) + g^*(u^*,x\dbs)\big]\cr
&\ge \infn_{u^* \in E^*}\big[\bra{x^* - u^*}{x\dbs} + \bra{u^*}{x\dbs}\big] = \bra{x^*}{x\dbs} = s_L(x^*,x\dbs).
\end{aligned}
\right\}
\end{equation}
Thus \eqref{TOUCHD7} follows from \eqref{TOUCHD4} and \Thm{CONJTOUCHthm}.   If $(x,x^*) \in E \times E^*$ then we obtain \eqref{TOUCHD3} by setting $x\dbs = \wh x$ in \eqref{TOUCHD6}, and \eqref{TOUCHD5} follows from
\eqref{TOUCHD4}, \eqref{TOUCHD7} and \Thm{Kthm}(b). 
\end{proof}
\begin{remark}\label{SZrem}
We do not assert in \eqref{TOUCHD4} that $h \in \PCLSCQ(E \times E^*)$.
\end{remark}
\Thm{STDthm}  below has applications to the classification of maximally monotone multifunctions.   See \cite[Theorems~14.2 and 16.2]{RLv7}.   \Thm{STDthm} can also be deduced from Voisei--Z\u{a}linescu \cite[Corollary~3.5,\ p.\ 1024]{VZ}.
\begin{theorem}[Sum theorem with domain constraints]\label{STDthm}
Let $S,T\colon\ E \toto E^*$ be closed, monotone and quasidense.   Then {\rm(a)$\lr$(b)$\lr$(c)$\lr$(d)}:
\smallbreak
\noindent
{\rm(a)}\enspace $D(S) \cap \intr\,D(T) \ne \emptyset$ or $\intr\,D(S) \cap D(T) \ne \emptyset$.
\smallbreak
\noindent
{\rm(b)}\enspace $\ts\bigcupn_{\lambda > 0}\lambda\big[D(S) - D(T)\big] = E$.
\smallbreak
\noindent
{\rm(c)}\enspace $\ts\bigcupn_{\lambda > 0}\lambda\big[\pi_1\,\dom\,\varphi_S - \pi_1\,\dom\,\varphi_T\big]$\quad is a closed subspace of $E$.
\smallbreak
\noindent
{\rm(d)}\enspace $S + T$\quad is closed, monotone and stably quasidense.
\end{theorem}
\begin{proof}
It is immediate \big(using \eqref{PHIS2}\big) that (a)$\lr$(b)$\lr$(c).   Now suppose that (c) is satisfied.  From \eqref{PHIS1} and \eqref{PHIS4}, we can apply \Lem{TOUCHDlem} with $f := \varphi_S$ and $g := \varphi_T$.   So, in this case, \eqref{TOUCHD1} gives
\smallbreak
\centerline{$h(x,x^*) := \infn_{\xi^* \in E^*}\big[\varphi_S(x,x^* - \xi^*) + \varphi_T(x,\xi^*)\big]$.}
\smallbreak
\noindent
Thus \eqref{TOUCHD5} is satisfied and, for all $(x,x^*) \in E \times E^*$, \eqref{TOUCHD3} is satisfied.    We now prove that
\begin{equation}\label{STD9}
\big\{E \times E^*|h^@ = q_L\big\} = G(S + T).
\end{equation}
To this end, first let $(y,y^*) \in E \times E^*$ and $h^@(y,y^*) = q_L(y,y^*) = \bra{y}{y^*}$. \eqref{TOUCHD3} now gives $u^* \in E^*$ such that\quad ${\varphi_S}^@(y,y^* - u^*) + {\varphi_T}^@(y,u^*) = \bra{y}{y^*}$.\quad From \eqref{PHIS3},\quad ${\varphi_S}^@(y,y^* - u^*) \ge \bra{y}{y^* - u^*}$\quand${\varphi_T}^@(y,u^*) \ge \bra{y}{u^*}$.\quad   Since $\bra{y}{y^* - u^*} + \bra{y}{u^*} = \bra{y}{y^*}$, in fact ${\varphi_S}^@(y,y^* - u^*) = \bra{y}{y^* - u^*}$ and ${\varphi_T}^@(y,u^*) = \bra{y}{u^*}$, and another application of \eqref{PHIS3} implies that $(y,y^* - u^*) \in G(S)$ and $(y,u^*) \in G(T)$, from which $(y,y^*) \in G(S + T)$.   Suppose, conversely, that $(y,y^*) \in G(S + T)$.   Then there exists $u^* \in E^*$ such that $(y,y^* - u^*) \in G(S)$ and $(y,u^*) \in G(T)$.   From \eqref{TOUCHD3} and \eqref{PHIS3},
\begin{align*}
h^@(y,y^*) &\le {\varphi_S}^@(y,y^* - u^*) + {\varphi_T}^@(y,u^*)\\
&= \bra{y}{y^* - u^*} + \bra{y}{u^*} = \bra{y}{y^*} \le h^@(y,y^*),
\end{align*}
thus $h^@(y,y^*) = \bra{y}{y^*} = q_L(y,y^*)$.   This completes the proof of \eqref{STD9}, and (d) follows by combining \eqref{STD9} and \eqref{TOUCHD5}.
\end{proof}
\begin{definition}[The Fitzpatrick extension]\label{FITZdef}
Let the notation be as in\break \Sec{EEsec} and $A$ be a closed, quasidense monotone subset of $E \times E^*$.   From\break \Thm{RLMAXthm}(a), $A$ is maximally monotone.   \Cor{PHICRITcor}\big((a)$\lr$(d)\big) and\break \Lem{EESLlem} imply that ${\Phi_A}^* \ge \qLt$ on $E^* \times E\dbs$.   We then write
\begin{equation}\label{FITZ1}
A^\F := \big\{E^* \times E\dbs|{\Phi_A}^* = \qLt\big\}.
\end{equation}
Let $(x,x^*) \in E \times E^*$.   Then, from \eqref{FAT} and \eqref{PHI4},
$(x,x^*) \in L^{-1}A^\F \iff {\Phi_A}^*L(x,x^*) = \qLt L(x,x^*) \iff {\Phi_A}^@(x,x^*) = \bra{x}{x^*} \iff (x,x^*) \in A$.   Thus
\begin{equation}\label{FITZ2}
L^{-1}A^\F = A,
\end{equation}
and so $A^\F$ is, in some sense, an extension of $A$ to $E^* \times E\dbs$. We will describe $A^\F$ as the {\em Fitzpatrick extension of $A$}.   It follows from this that $A^\F \ne \emptyset$, and so \Lem{Llem}(a) (with $B := E^* \times E\dbs$ and $f := {\Phi_A}^*$) implies that
\begin{equation}\label{AFMON}
A^\F\hbox{ is monotone.}
\end{equation}
In fact, as we shall see in \Thm{AFMAXthm}, $A^\F$ is always maximally monotone, but we do not need this result at the moment.   We digress briefly to the multifunction versions of the above concepts.   If $S\colon\ E \toto E^*$ is closed, monotone and quasidense then
\begin{equation}\label{FITZ3}
{\varphi_S}^* \ge \qLt\hbox{ on }E^* \times E\dbs.
\end{equation}
We define the multifunction $S^\F\colon\ E^* \toto E\dbs$ so that $G(S^\F) = G(S)^\F$.   Thus  $x\dbs \in S^\F(x^*)$ exactly when ${\varphi_S}^*(x^*,x\dbs) = \bra{x^*}{x\dbs}$.   It also follows from \eqref{FITZ2} that
\begin{equation}\label{FITZ4}
x^* \in S(x) \iff \wh x \in S^\F(x^*).
\end{equation}
Finally, $S^\F$ is monotone.   We will continue our development of the theory of the Fitzpatrick extension in \Sec{FITZCHARsec}.  
\end{definition}
By interchanging the order of the variables in the statement of \Lem{SZlem}, we can prove the following result in a similar fashion:
\begin{lemma}\label{SZBISlem}
Let $f,g \in \PCLSC(X \times Y)$.   For all $(x,y) \in X \times Y$, let
$$h(x,y) := \infn_{u \in X}\big[f(x - u,y) + g(u,y)\big] > -\infty.$$
Suppose that
$$\ts\bigcupn_{\lambda > 0}\lambda\big[\pi_2\,\dom\,f - \pi_2\,\dom\,g\big]\hbox{ is a closed subspace of }Y.$$
Then $h \in \PC(X \times Y)$ and, for all $(x^*,y^*) \in X^* \times Y^*$,
$$h^*(x^*,y^*) = \minn_{v^* \in Y^*}\big[f^*(x^*,y^* - v^*) + g^*(x^*,v^*)\big].$$
\end{lemma}
\begin{lemma}\label{TOUCHRlem}
Let $f,g \in \PCLSCQ(E \times E^*)$ be touching.   For all $(x,x^*) \in E \times E^*$, let
\begin{equation}\label{TOUCHR1}
h(x,x^*) := \infn_{\xi \in E}\big[f(x - \xi,x^*) + g(\xi,x^*)\big].
\end{equation}
Suppose that
\begin{equation}\label{TOUCHR2}
\ts\bigcupn_{\lambda > 0}\lambda\big[\pi_2\dom\,f - \pi_2\dom\,g\big]\ \hbox{is a closed linear subspace of}\ E^*.\end{equation}
Then
\begin{equation}\label{TOUCHR4}
h \in \PCQ(E \times E^*),
\end{equation}
\begin{equation}\label{TOUCHR8}
h\hbox{ is touching},
\end{equation}
for all $(x,x^*) \in E \times E^*$,
\begin{equation}\label{TOUCHR3}
h^@(x,x^*) = \minn_{z\dbs \in E\dbs}\big[f^*(x^*,\wh x - z\dbs) + g^*(x^*,z\dbs)\big] \ge \bra{x}{x^*},
\end{equation}
and
\begin{equation}\label{TOUCHR7}
\big\{E \times E^*|h^@ = q_L\big\}\hbox{ is closed, monotone and stably quasidense.}
\end{equation}
\end{lemma}
\begin{proof}
\eqref{TOUCHR2} implies that $\pi_2\,\dom\,g \cap \pi_2\,\dom\,f \ne \emptyset$, and so there exists\break $(x_0,y_0,z_0^*) \in E \times E \times E^*$ such that $f(x_0,z_0^*) \in \RR$ and $g(y_0,z_0^*) \in \RR$.   By hypothesis, $f \ge q_L$ and $g \ge q_L$ on $E \times E^*$.  Then, for all $(x,x^*) \in E \times E^*$,\break $h(x,x^*) \ge \infn_{\xi \in E}\big[\bra{x - \xi}{x^*} + \bra{\xi}{\xi^*}\big] = \bra{x}{x^*}$\quand $h(x_0 + y_0,z_0^*) \le\break f(x_0,z_0^*) + g(y_0,z_0^*) < \infty$, consequently \eqref{TOUCHR4} is satisfied.   From \Thm{CONJTOUCHthm} and \Lem{EESLlem}, for all $(x^*,x\dbs) \in E^* \times E\dbs$, 
\smallbreak
\centerline{${f}^*(x^*,x\dbs) \ge \bra{x^*}{x\dbs}\quand {g}^*(x^*,x\dbs) \ge \bra{x^*}{x\dbs}$.}
\smallbreak
\noindent
Thus \Lem{SZBISlem} (with $X := E$ and $Y := E^*$) and \Lem{EESLlem} imply that for all $(x^*,x\dbs) \in E^* \times E\dbs$,
\begin{equation}\label{TOUCHR6}
\left.
\begin{aligned}
&h^*(x^*,x\dbs) \\
&= \minn_{z\dbs \in E\dbs}\big[f^*(x^*,x\dbs - z\dbs) + g^*(x^*,z\dbs)\big]\cr
&\ge \infn_{z\dbs \in E\dbs}\big[\bra{x^*}{x\dbs - z\dbs} + \bra{x^*}{z\dbs}\big] = \bra{x^*}{x\dbs} = s_L(x^*,x\dbs).
\end{aligned}
\right\}
\end{equation}
Thus \eqref{TOUCHR8} follows from \eqref{TOUCHR4} and \Thm{CONJTOUCHthm}.
If $(x,x^*) \in E \times E^*$ then we obtain \eqref{TOUCHR3} by setting $x\dbs = \wh x$ in \eqref{TOUCHR6}, and \eqref{TOUCHR7} follows from
\eqref{TOUCHR4}, \eqref{TOUCHR8} and \Thm{Kthm}(b).
\end{proof}
If $S,T\colon\ E \toto E^*$ then the {\em parallel sum}, $S \parallel T\colon\ E \toto E^*$ is defined to be $(S^{-1} + T^{-1})^{-1}$.   \Thm{STRthm} below has applications to the classification of maximally monotone multifunctions.   See \cite[Theorems~14.2 and 18.4]{RLv7}.
\begin{theorem}[Sum theorem with range constraints]\label{STRthm}
Let $S,T\colon\ E \toto E^*$ be closed, monotone and quasidense.   Then {\rm(a)$\lr$(b)$\lr$(c)$\lr$(d)$\lr$(e)}:
\smallbreak
\noindent
{\rm(a)}\enspace $R(S) \cap \intr\,R(T) \ne \emptyset$ or $\intr\,R(S) \cap R(T) \ne \emptyset$.
\smallbreak
\noindent
{\rm(b)}\enspace $\ts\bigcupn_{\lambda > 0}\lambda\big[R(S) - R(T)\big] = E^*$.
\smallbreak
\noindent
{\rm(c)}\enspace $\ts\bigcupn_{\lambda > 0}\lambda\big[\pi_2\,\dom\,\varphi_S - \pi_2\,\dom\,\varphi_T\big]$\quad is a closed subspace of $E^*$.
\smallbreak
\noindent
{\rm(d)}\enspace Define the multifunction $P\colon\ E \toto E^*$ by $P(y) := (S^\F + T^\F)^{-1}(\wh y)$.   Then $P$ is closed, monotone and stably quasidense.
\smallbreak
\noindent
{\rm(e)}\enspace If, further, $G(T)^\F = L\big(G(T)\big)$ then $S \parallel T$ is  closed, monotone and stably quasidense.
\end{theorem}
\begin{proof}
It is immediate \big(using \eqref{PHIS2}\big) that (a)$\lr$(b)$\lr$(c).   Now suppose that (c) is satisfied.   From \eqref{PHIS1} and \eqref{PHIS4}, we can apply \Lem{TOUCHRlem} with $f := \varphi_S$ and $g := \varphi_T$.   So, in this case, \eqref{TOUCHR1} gives
\smallbreak
\centerline{$h(x,x^*) := \infn_{\xi \in E}\big[\varphi_S(x - \xi,x^*) + \varphi_T(\xi,x^*)\big]$.}
\smallbreak
\noindent
Thus \eqref{TOUCHR7} is satisfied and, for all $(x,x^*) \in E \times E^*$, \eqref{TOUCHR3} is satisfied.    Let $(y,y^*) \in E \times E^*$.   We prove that
\begin{equation}\label{STR8}
h^@(y,y^*) = q_L(y,y^*) \iff y^* \in P(y).
\end{equation}
To this end, first let $h^@(y,y^*) = q_L(y,y^*) = \bra{y}{y^*}$. \eqref{TOUCHR3} now gives $z\dbs \in E\dbs$ such that\quad ${\varphi_S}^*(y^*,\wh y - z\dbs) + {\varphi_T}^*(y^*,z\dbs) = \bra{y}{y^*}$.\quad We know from \eqref{FITZ3} that\quad ${\varphi_S}^*(y^*,\wh y - z\dbs) \ge \bra{y^*}{\wh y - z\dbs}$ and ${\varphi_T}^*(y^*,z\dbs) \ge \bra{y^*}{z\dbs}$.\quad   Since\break $\bra{y^*}{\wh y - z\dbs} + \bra{y^*}{z\dbs} = \bra{y}{y^*}$,\quad in fact\quad ${\varphi_S}^*(y^*,\wh y - z\dbs) = \bra{y^*}{\wh y - z\dbs}$ and ${\varphi_T}^*(y^*,z\dbs) = \bra{y^*}{z\dbs}$,\quad that is to say,\quad $\wh y - z\dbs \in S^\F(y^*)$ and $z\dbs \in T^\F(y^*)$, and so $y^* \in P(y)$.   Suppose, conversely, that $y^* \in P(y)$.   Then there exists $z\dbs \in T^\F(y^*)$ such that $\wh y - z\dbs \in S^\F(y^*)$    From \eqref{TOUCHR3},
\begin{align*}
h^@(y,y^*) &\le {\varphi_S}^*(y^*,\wh y - z\dbs) + {\varphi_T}^*(y^*,z\dbs)\\
&= \bra{y^*}{\wh y - z\dbs} + \bra{y^*}{z\dbs} = \bra{y}{y^*} \le h^@(y,y^*).
\end{align*}
Thus $h^@(y,y^*) = \bra{y}{y^*} = q_L(y,y^*)$.  This completes the proof of \eqref{STR8}, and (d) follows by combining \eqref{STR8} and \eqref{TOUCHR7}.
\smallbreak
Suppose, finally, that $G(T)^\F = L\big(G(T)\big)$.  We will prove $P = S \parallel T$, and (e) then follows from (d).   To this end, first let $y^* \in P(y)$.   Then we can choose $z\dbs \in T^\F(y^*)$ such that $\wh y - z\dbs \in S^\F(y^*)$. Now $(y^*,z\dbs) \in G\big(T^\F\big) = L\big(G(T)\big)$, and so there exists $(z,z^*) \in G(T)$ such that $(y^*,z\dbs) = (z^*,\wh z)$, from which $z\dbs = \wh z$, $z \in T^{-1}(z^*) = T^{-1}(y^*)$ and $\wh y - \wh z \in S^\F(y^*)$.   But then \eqref{FITZ4} implies that $y - z \in S^{-1}y^*$.   Thus $y = (y - z) + z \in S^{-1}y^* + T^{-1}y^*$, from which $y^* \in (S \parallel T)(y)$.   If, conversely, $y^* \in (S \parallel T)(y)$ then there exists $z \in E$ such that $y^* \in S(y - z)$ and  $y^* \in T(z)$.   From \eqref{FITZ4}, $\wh y - \wh z \in S^\F(y^*)$ and  $\wh z \in T^F(y^*)$.   Thus $\wh y \in (S^\F + T^\F)(y^*)$, and so $y^* \in  P(y)$.   This completes the proof of (e). 
\end{proof}
\section{Closed $L$--positive linear subspaces}\label{LINsec}
In this section, we suppose that $(B,L)$ is a Banach SN space and  $A$ is a closed $L$--positive linear subspace of $B$.   The analysis in this section differs from that in \cite[Section~9]{POLAR} in that $(B,L)$ is {\em not} required to have a Banach SN dual.   We also point out the novel use of the $r_L$--density of subdifferentials to prove results on linear subspaces.  We define the function\quad $k\colon\ B \to \,\rbar$\quad by\quad $k := q_L + \I_A.$\quad We write $A^0$ for the linear subspace  $\big\{b^* \in B^*\colon\ \bra{A}{b^*} = \{0\}\big\}$ of $B^*$.   $A^0$ is the ``polar subspace of $A$''.   The significance of $A^0$ lies in the following lemma:  
\begin{lemma}\label{QCSTARlem}
$k \in \PCLSCQ(B)$, $\big\{B|k = q_L\big\} = A$ and
$$\partial k(b) = \begin{cases}Lb + A^0&(\hbox{if}\ b \in A);\cr \emptyset&(\hbox{if}\  b \in B \setminus A).\end{cases}$$
\end{lemma}
\begin{proof}
$k$ is obviously proper.   Suppose that $b,c \in A$ and $\lambda \in \,]0,1[\,$.   Then
\begin{align*}
0 \le  \lambda(1 - \lambda)q_L(b - c) 
&= \lambda q_L(b) + (1 - \lambda)q_L(c) - q_L\big(\lambda b + (1 - \lambda)c\big)\\
&= \lambda k(b) + (1 - \lambda)k(c) - k\big(\lambda b + (1 - \lambda)c\big).
\end{align*}
This implies the convexity of $k$.  \big(See \cite[Lemma 19.7, pp.\ 80--81]{HBM}.\big)   Since $q_L$ is continuous and $A$ is closed, $k$ is lower semicontinuous.   It is now obvious that $k \in \PCLSCQ(B)$ and $\big\{B|k = q_L\big\} = A$. Since $\partial k(b) = \emptyset$ if \quad $b \in B \setminus \dom\,k$,\quad that is to say,\quad $b \in B \setminus A$,\quad it only remains to show that\quad $\partial k(b) = Lb + A^0$\quad if $b \in A$.\quad So suppose that $b \in A$.   Then, since $c := a - b$ runs through $A$ as $a$ runs through $A$ and then\quad $k(b) - k(c + b) = q_L(b) - q_L(c + b) = -\bra{c}{Lb} - q_L(c)$,
\begin{align*}
b^* \in \partial k(b)
&\iff \supn_{a \in A}\big[k(b) + \bra{a - b}{b^*} - k(a)\big] \le 0\\
&\iff \supn_{c \in A}\big[k(b) + \bra{c}{b^*} - k(c + b)\big] \le 0\\
&\iff \supn_{c \in A}\big[\bra{c}{b^* - Lb} - q_L(c)\big] \le 0.
\end{align*}
Since $q_L(c) \ge 0$ for all $c \in A$, this is trivially satisfied if $b^* \in Lb + A^0$.   On the other hand, if $b^* \in \partial k(b)$ then it follows from the above that, for all $c \in A$ and $\lambda \in \RR$,\quad $\lambda \bra{c}{b^* - Lb} - \lambda^2q_L(c) = \bra{\lambda c}{b^* - Lb} - q_L(\lambda c) \le 0$.\quad Thus, from the standard quadratic arguments, for all $c \in A$, $\bra{c}{b^* - Lb} = 0$. This is equivalent to the statement that $b^* \in Lb + A^0$.
\end{proof}
\begin{lemma}\label{LINlem}
Let $b \in B$.   Then $\inf r_L(A - b) \le \sup s_L\big(A^0\big)$.
\end{lemma}
\begin{proof}
Let $\eps > 0$.  Define $\wh L\colon\ B \times B^* \to B^* \times B\dbs$ by $\wh L(b,b^*) := \big(b^*,\wh{b}\big)$.   From \Lem{QCSTARlem} and \Thm{RTRthm}, there exist $a \in A$ and $d^* \in \partial k(a) = La + A^0$ such that
$$\half\|a - b\|^2 + \half\|d^* - Lb\|^2 + \bra{a - b}{d^* - Lb} = r_{\wh L}\big((a,d^*) - (b,Lb)\big) < \eps.$$
We write $c = b - a$ and $b^* = d^* - La \in A^0$.    Then $d^* - Lb = b^* - Lc$, from which $ \half\|c\|^2 + \half\|b^* - Lc\|^2 - \bra{c}{b^* - Lc} < \eps$, which can be rewritten
$$\half\|c\|^2 <  \bra{c}{b^*} - 2q_L(c) - \half\|b^* - Lc\|^2 + \eps.$$
It follows from \eqref{SL1} that
$$r_L(a - b) = \half\|c\|^2 + q_L(c) < \bra{c}{b^*} - q_L(c) - \half\|b^* - Lc\|^2 + \eps \le s_L(b^*) + \eps.$$
Since $b^* \in A^0$, this gives the required result.
\end{proof}
\Thm{LINthm} and \Cor{LINcor} will be used in \Thm{LINMONthm}.
\begin{theorem}\label{LINthm}
$A$ is $r_L$--dense in $B$ if, and only if, $\sup s_L\big(A^0\big) \le 0$.
\end{theorem}
\begin{proof}
Suppose first that $A$ is $r_L$--dense in $B$.   From \Lem{QCSTARlem}, $k \in \PCLSCQ(B)$ and $\big\{B|k = q_L\big\} = A$.   From the $L$--positivity of $A$,  \Cor{RLEQcor}\big((a)$\lr$(b)\big) and \Thm{CONJTOUCHthm},  for all $b^* \in A^0$,
$$0 = \supn_A\big[- q_L\big] = \supn_A\big[b^* - q_L\big] = \supn_B\big[b^* - k\big] = k^*(b^*) \ge s_L(b^*).$$
Thus $\sup s_L\big(A^0\big) \le 0$.   If, conversely, $\sup s_L\big(A^0\big) \le 0$, then it is immediate from \Lem{LINlem} that $A$ is $r_L$--dense in $B$.
\end{proof}
\begin{corollary}\label{LINcor}
Let $c^* \in B^*$ and $\sup s_L\big(A^0 + \lin\{c^*\}\big) \le 0$.   Then $c^* \in A^0$.
\end{corollary}
\begin{proof}
Suppose that $c^* \not\in A^0$.   Let $Z = \big\{b \in B\colon\ \bra{b}{c^*} = 0\big\}$.   It is well known that $Z^0 =  \lin\{c^*\}$.   Since $c^* \not\in A^0$, there exists $a \in A \setminus Z$, and so the fact that $Z$ has codimension 1 implies that $A - Z = B$, that is to say $\dom\,\I_A - \dom\,\I_Z = B$.   From the Attouch--Brezis formula for the subdifferential of a sum,
\begin{align*}
(A \cap Z)^0
&= \partial\big(\I_{A \cap Z}\big)(0) = \partial\big(\I_A + \I_Z\big)(0)\\
&= \partial\I_A(0) +  \partial\I_Z(0) = A^0 + Z^0 = A^0 + \lin\{c^*\}.
\end{align*}
Thus, by assumption, $\sup s_L\big((A \cap Z)^0\big) \le 0$.   Since $A \cap Z \subset A$, $A \cap Z$ is also a closed $L$--positive linear subspace of $B$, thus \Thm{LINthm} and \Lem{RLMAXlem} imply that $A \cap Z$ is maximally $L$--positive.   Since $A \cap Z \subset A$, it follows from this that $A \cap Z = A$, and so $A \subset Z$, which gives $c^* \in Z^0 \subset A^0$, a contradiction.   
\end{proof}
\begin{remark}\label{LINrem}
One can use \cite[Lemma~2.2, p.\ 260--261]{POLAR} instead of the Attouch--Brezis formula in the proof of \Cor{LINcor}.
\end{remark}
\section{Monotone linear relations and operators}\label{LINMONsec}
In this section, we suppose that $A$ is a linear subspace of $E \times E^*$ (commonly called a {\em linear relation}).  The {\em adjoint linear subspace}, $A^\T$, of $E\dbs \times E^*$, is defined by:
$$(y\dbs,y^*) \in A^\T \iff \hbox{for all}\ (s,s^*) \in A,\ \bra{s}{y^*} = \bra{s^*}{y\dbs}.$$
This definition goes back at least to Arens in \cite{ARENS}.   (We use the notation ``$A^\T$'' rather than the more usual ``$A^*$'' to avoid confusion with the dual space of $A$.)   It is clear that
$$(y\dbs,y^*) \in A^\T \iff (y^*,-y\dbs) \in A^0.$$  
\Thm{LINMONthm} below extends the result obtained by combining Bauschke,\break Borwein, Wang and Yao \cite[Theorem 3.1((iii)$\lr$(ii))]{BBWYLIN} and \cite[Proposition 3.1]{BBWYBB}, which in turn extends the result proved in the reflexive case by Brezis and Browder in \cite{BB}.   \Cor{LINMONcor} follows indirectly from Bauschke and Borwein, \cite[Theorem~4.1\big((iii)$\ifff$(v)\big), pp.\ 10--12]{BABO}.   \Thm{HTthm} provides more examples of maximally monotone linear operators that are not quasidense.   These examples can also be derived from the decomposition results of Bauschke and Borwein, \cite[Theorem~4.1\big((v)$\ifff$(vi)\big), pp.\ 10--12]{BABO}.
\begin{theorem}\label{LINMONthm}
Suppose that $A$ is a monotone closed linear subspace of $E \times E^*$.   Then $A$ is quasidense if, and only if, $A^\T$ is a monotone subspace of $E\dbs \times E^*$ if, and only if, $A^\T$ is a maximally monotone subspace of $E\dbs \times E^*$.
\end{theorem}
\begin{proof}
From \Thm{LINthm} and \Lem{EESLlem}, $A$ is quasidense if, and only if, for all $(x^*,x\dbs) \in A^0$,  $\bra{x^*}{x\dbs} \le 0$, that is to say, 
$$\all\ (y\dbs,y^*) \in A^\T,\quad \bra{y^*}{-y\dbs} \le 0.$$
This is clearly equivalent to the statement that $A^\T$ is a monotone subspace of $E\dbs \times E^*$.   The second equivalence is immediate from \Cor{LINcor}. 
\end{proof}
%
\begin{corollary}\label{LINMONcor}
Suppose that $S\colon E \to E^*$ is a monotone linear operator.   Then $S$ is quasidense if, and only if, the adjoint linear operator $S^\T\colon E\dbs \to E^*$ is  monotone.\end{corollary}
\begin{proof}
This is immediate from \Thm{LINMONthm} and the observation that\break $G(S^\T) = G(S)^\T$.
\end{proof}
\begin{example}[Heads and tails]\label{HTex}
We defined the {\em tail operator}, $T$, in\break \Ex{TAILex}.  We define the {\em head operator}\quad $H\colon\ \ell_1 \mapsto \ell_\infty = E^*$\quad by\break $(Hx)_n = \ts\sum_{k \le n}x_k$.   Using the notation of \Ex{TAILex}, for all $x \in \ell_1$,
\begin{equation}\label{HT1}
\ts\bra{x}{Hx} = \sum_{n \ge 1}x_n\sum_{k \le n}x_k = \sum_{k \ge 1}x_k\sum_{n \ge k}x_n = \bra{x}{Tx}.
\end{equation}
If $\lambda,\mu \in \RR$, $\lambda + \mu \ge 0$ and $S := \lambda T + \mu H$ then,  from \eqref{HT1}, $S$ is monotone.    Since $S$ is linear and has full domain, $S$ is maximally monotone.   In \Thm{HTthm} below, we find for which values of $\lambda$ and $\mu$ (with $\lambda + \mu \ge 0$) $S$ is quasidense.
\end{example}
\begin{theorem}[The theorem of the two quadrants]\label{HTthm}
Let $\lambda,\mu \in \RR$, $\lambda + \mu \ge 0$ and $S := \lambda T + \mu H$.   Then $S$ is quasidense if, and only if, $\lambda - \mu \le 0$.   In particular, let $G\colon\ \ell_1 \mapsto \ell_\infty = E^*$ be {\em Gossez's operator}, defined by $G := T - H$.   Then $G$ is not quasidense, but $-G$ is quasidense.
\end{theorem}
\begin{proof}
See \cite[Theorem~10.7]{RLv7}.    
\end{proof}
\section{Negative alignment conditions}\label{BRsec}
The material in this section was initially motivated by a result proved for reflexive spaces by Torralba in \cite[Proposition 6.17]{TORR} and extended to maximally monotone sets of type (D) by Revalski--Th\'era in \cite[Corollary~3.8, p. 513]{RT}.   In \Thm{NIBRthm}, we shall give a criterion for a closed monotone set to be quasidense in terms of {\em negative alignment pairs}, which are defined below, though the main result of this section is \Thm{MFBRthm}.   \Thm{MFBRthm}(c) is a version of the Br{\o}ndsted--Rockafellar theorem for closed monotone quasidense sets.   See \cite[Section~8, pp.\ 274--280]{BR} for a more comprehensive discussion of the history of this kind of result.   In this section we shall give complete details of proofs only if they differ in some significant way from those in \cite{BR}.
\begin{definition}\label{NAPAIRdef}
Let $A \subset E \times E^*$ and $\rho,\ \sigma \ge 0$.   We say that $(\rho,\sigma)$ is a {\em negative alignment pair} for $A$ with respect to $(w,w^*)$ if there exists a sequence $\big\{(s_n,s_n^*)\big\}_{n \ge 1}$ of elements of $A$ such that
$$\lim_{n \to \infty}\|s_n - w\| = \rho,\quad
\lim_{n \to \infty}\|s_n^* - w^*\| =\sigma \quand
\lim_{n \to \infty}\bra{s_n - w}{s_n^* - w^*} = -\rho\sigma.$$
\end{definition}
\begin{lemma}\label{MFBRlem}
Let $A$ be s closed subset of $E \times E^*$, $(w,w^*) \in E \times E^*$, $\alpha,\beta > 0$, $\tau \ge 0$ and $(\alpha\tau,\beta\tau)$ be a negative alignment pair for $A$ with respect to $(w,w^*)$.
\smallbreak
\noindent
{\rm(a)}\enspace If $(w,w^*) \in E \times E^* \setminus A$ then $\tau > 0$.
\smallbreak
\noindent
{\rm(b)}\enspace If $\inf_{(s,s^*) \in A}\bra{s - w}{s^* - w^*} > -\alpha\beta$ 
then $\tau < 1$.
\end{lemma}
\begin{proof}
From \Def{NAPAIRdef}, there exists a sequence $\big\{(s_n,s_n^*)\big\}_{n \ge 1}$ of elements of $A$ such that\quad
$\lim_{n \to \infty}\|s_n - w\| = \alpha\tau$,\quad
$\lim_{n \to \infty}\|s_n^* - w^*\| = \beta\tau$ \quand
$\lim_{n \to \infty}\bra{s_n - w}{s_n^* - w^*} = -\alpha\beta\tau^2$.
\par
(a)\enspace If $\tau = 0$ then, since $A$ is closed, $(w,w^*) \in A$.
\smallbreak
(b)\enspace Since
$-\alpha\beta\tau^2 = \limn_{n \to \infty}\bra{s_n - w}{s_n^* - w^*} \ge \infn_{(s,s^*) \in A}\bra{s - w}{s^* - w^*} > -\alpha\beta$,
it follows that $\tau < 1$.
\end{proof}
Our next result contains a uniqueness theorem for negative alignment pairs for the case when $A$ is monotone.   The proof can be found in \cite[Theorem~8.4(b), p.\ 276]{BR}.
\begin{lemma}\label{UNIlem}
Let $A$ be a monotone subset of $E \times E^*$, $(w,w^*) \in E \times E^*$ and $\alpha,\beta > 0$.   Then there exists at most one value of $\tau \ge 0$ such that $(\alpha\tau,\beta\tau)$ is a negative alignment pair for $A$ with respect to $(w,w^*)$.
\end{lemma}
We now give our main result on the existence of negative alignment pairs, and some simple consequences.  We refer the reader to \Rem{NIrem} for more discussion on some of the issues raised by these results.
\begin{theorem}\label{MFBRthm}
Let $A$ be closed, monotone, quasidense subset of $E \times E^*$, $(w,w^*) \in E \times E^*$ and
$\alpha,\beta > 0$.   Then:
\smallbreak
\noindent
{\rm(a)}\enspace There exists a unique value of $\tau \ge 0$ such that
$(\alpha\tau,\beta\tau)$ is a\break negative alignment pair for $A$ with
respect to $(w,w^*)$.
\smallbreak
\noindent
{\rm(b)}\enspace If $(w,w^*) \not\in A$ then there exists a sequence $\big\{(s_n,s_n^*)\big\}_{n \ge 1}$ of elements of $A$ such that, for all $n \ge 1$, $s_n \ne w$, $s_n^* \ne w^*$,
\begin{equation}\label{ONE2}
\lim_{n \to \infty}\frac{\|s_n - w\|}{\|s_n^* - w^*\|} = \frac\alpha\beta\quand
\lim_{n \to \infty}\frac{\bra{s_n - w}{s_n^* - w^*}}{\|s_n - w\|\|s_n^* - w^*\|} = -1.
\end{equation}
{\rm(c)}\enspace If $\inf_{(s,s^*) \in A}\bra{s - w}{s^* - w^*} > -\alpha\beta$ then there exists $(s,s^*) \in A$ such that $\|s - w\| < \alpha$ and $\|s^* - w^*\| < \beta$.     If, further, $(w,w^*) \not\in A$ then there exists a sequence $\big\{(s_n,s_n^*)\big\}_{n \ge 1}$ of elements of $A$ such that, for all $n \ge 1$, $s_n \ne w$, $s_n^* \ne w^*$, $\|s_n - w\| < \alpha$, $\|s_n^* - w^*\| < \beta$, and \eqref{ONE2} is satisfied.
\smallbreak
\noindent
{\rm(d)}\enspace $\overline{\pi_1A} = \overline{\pi_1\dom\,\Phi_A}$ and $\overline{\pi_2A} = \overline{\pi_2\dom\,\Phi_A}$. Consequently, the sets $\overline{\pi_1A}$ and $\overline{\pi_2A}$ are convex.
\end{theorem}
\begin{proof}
Define $\Delta$ as in \Lem{STABlem} and let $(u,u^*) := \Delta(w,w^*)$.  From \Lem{STABlem}, $\Delta(A)$ is closed, monotone and stably quasidense, and so there exists a {\em bounded} sequence $\big\{(t_n,t_n^*)\big\}_{n \ge 1}$ of elements of $\Delta(A)$ such that
\begin{equation}\label{ONE1}
\left.\begin{aligned}
0 &= \limn_{n \to \infty}r_L(t_n - u,t_n^* - u^*)\\
&= \limn_{n \to \infty}\big(\half\|t_n - u\|^2 + \half\|t_n^* - u^*\|^2 + \bra{t_n - u}{t_n^* - u^*}\big)\\
&\ge \limsupn_{n \to \infty}\half\big(\|t_n - u\| - \|t_n^* - u^*\|\big)^2 \ge 0.
\end{aligned}\right\}
\end{equation}
Thus\quad $\limn_{n \to \infty}\big(\|t_n - u\| - \|t_n^* - u^*\|) = 0$.\quad   Since $\big\{\|(t_n - u\|\big\}_{n \ge 1}$ is bounded in $\RR$, passing to an appropriate subsequence, there exists $\tau \in \RR$ such that $\tau \ge 0$ and\quad $\lim_{n \to \infty}\|t_n - u\| = \tau$,\quad from which\quad $\lim_{n \to \infty}\|t_n^* - u^*\| = \tau$\quad also.   From \eqref{ONE1},\quad $\lim_{n \to \infty}\bra{t_n - u}{t_n^* - u^*} = -\half\tau^2 - \half\tau^2 = -\tau^2$.\quad  For all $n \ge 1$, let $(s_n,s_n^*) := (\alpha t_n,\beta t_n^*) \in A$.   Then,\quad since $(w,w^*) = (\alpha u,\beta u^*)$,\quad
$\lim_{n \to \infty}\|s_n - w\| = \alpha\tau$,\quad $\lim_{n \to \infty}\|s_n^* - w^*\| = \beta\tau$ \quand $\lim_{n \to \infty}\bra{s_n - w}{s_n^* - w^*} = -\alpha\beta\tau^2$.   Thus $(\alpha\tau,\beta\tau)$ is a negative alignment pair for $A$ with respect to $(w,w^*)$, and the ``uniqueness'' is immediate from \Lem{UNIlem}.   This completes the proof of (a).
\smallbreak
(b) follows from (a) and \Lem{MFBRlem}(a).
\smallbreak
(c) follows from (a) and \Lem{MFBRlem}(a,b).
\smallbreak
(d)\enspace If $w \in \pi_1\dom\,\Phi_A$ then there exists $w^* \in E^*$ such that $\Phi_A(w,w^*) < \infty$ thus, from \eqref{PHI1},
\begin{align*}
\inf_{(s,s^*) \in A}\bra{s - w}{s^* - w^*}
&= \bra{w}{w^*} - \sup_{(s,s^*) \in A}\big[\bra{s}{w^*} + \bra{w}{s^*} - \bra{s}{s^*}\big]\\
&= \bra{w}{w^*} - \Phi_A(w,w^*) > -\infty.
\end{align*}
Let $n \ge 1$ and $\beta > -n\inf_{(s,s^*) \in A}\bra{s - w}{s^* - w^*}$.   (c) now gives $(s,s^*) \in A$ such that $\|s - w\| < 1/n$.   Consequently, $w \in \overline{\pi_1A}$.   Thus we have proved that $\pi_1\dom\,\Phi_A \subset \overline{\pi_1A}$.   On the other hand, from \Thm{RLMAXthm}(a) and \eqref{PHI2},\break $\pi_1A \subset \pi_1\dom\,\Phi_A$, and so $\overline{\pi_1A} = \overline{\pi_1\dom\,\Phi_A}$.   Similarly, $\overline{\pi_2A} = \overline{\pi_2\dom\,\Phi_A}$.   The convexity of the sets $\overline{\pi_1A}$ and $\overline{\pi_2A}$ now follows immediately.
\end{proof}
\begin{remark}\label{MFBRrem}
In multifunction terms, \Thm{MFBRthm}(d) implies that if\break $S\colon E \toto E^*$ is closed monotone and quasidense then $\overline{D(S)}$ and $\overline{R(S)}$ are both convex.
\end{remark}
\begin{theorem}[A negative alignment criterion for quasidensity]\label{NIBRthm}
Let $A$ be a closed, monotone, subset of $E \times E^*$.   Then $A$ is quasidense if, and only if, for all $(w,w^*) \in E \times E^*$, there exists $\tau \ge 0$ such that $(\tau,\tau)$ is a negative alignment pair for $A$ with respect to $(w,w^*)$.\end{theorem}
\begin{proof}
Suppose first that, for all $(w,w^*) \in E \times E^*$, there exists $\tau \ge 0$ such that $(\tau,\tau)$ is a negative alignment pair for $A$ with respect to $(w,w^*)$.   Then, for all $(w,w^*) \in E \times E^*$, \Def{NAPAIRdef}, provides a sequence $\big\{(s_n,s_n^*)\big\}_{n \ge 1}$ of elements of $A$ such that
$$\lim_{n \to \infty}\|s_n - w\| = \tau,\quad \lim_{n \to \infty}\|s_n^* - w^*\| = \tau \quand
\lim_{n \to \infty}\bra{s_n - w}{s_n^* - w^*} = -\tau^2.$$
But then
\begin{align*}
\lim_{n \to \infty}r_L\big((s_n,s_n^*) &- (w,w^*)\big)\\
&= \lim_{n \to \infty}\big[\half\|s_n - w\|^2 + \half\|s_n^* - w^*\|^2 + \bra{s_n - w}{s_n^* - w^*}\big]\\
&= \half\tau^2 + \half\tau^2 - \tau^2 = 0.
\end{align*}
So $A$ is quasidense.   (The above analysis does not use the assumption that $A$ is closed or monotone.)  The converse is immediate from \Thm{MFBRthm}(a) with $\alpha = \beta = 1$.
\end{proof}
\begin{remark}\label{NIrem}
Using results on maximally monotone multifunctions of type (NI), \cite[Remark~11.4, p.\ 283]{BR} shows that the conclusion of \Thm{MFBRthm}(c) may, indeed, be true even if $A$ is {\em not} quasidense, and \cite[Example~11.5, p.\ 283--284]{BR} shows that if $A$ is not quasidense then the conclusion of \Thm{MFBRthm}(c) may fail.   In both these examples, $A$ is the graph of a single--valued, continuous linear map.   \Thm{MFBRthm}(d) implies that the closures of the domain and the range of a maximally monotone multifunction of type (NI) are both convex.   This result was first proved by Zagrodny in \cite{ZAGRODNY}, before it was known that such multifunctions of are always of type (ED).   See \Rem{ZAGrem}.
\end{remark}
\begin{definition}\label{ANAdef}
Let $A$ be a maximally monotone subset of $E \times E^*$.   Then $A$ is of {\em type (ANA)} if, whenever $(w,w^*) \in E \times E^* \setminus A$, there exists $(s,s^*) \in A$ such that $s \ne w$, $s^* \ne w^*$, and
$$\frac{\bra{s - w}{s^* - w^*}}{\|s - w\|\|s^* - w^*\|}\ \hbox{is as near
as we please to}\ {-}1.$$
(ANA) stands for ``almost negative alignment''.   See \cite[Section~9, pp. 280--281]{BR} for more discussion about this concept.
\end{definition}  
\begin{theorem}\label{ANAthm}
Let $A$ be closed, monotone, quasidense subset of $E \times E^*$.   Then $A$ is of type (ANA).
\end{theorem}  
\begin{proof}
This is immediate from \Thm{MFBRthm}(b).
\end{proof}
\begin{problem}
Does there exist a maximally monotone set that is not of type (ANA)?   The tail operator does {\em not} provide an example, because it was proved in Bauschke--Simons, \cite[Theorem 2.1, pp. 167--168]{BAUSIM} that {\em if $S\colon\ E \to E^*$ is monotone and linear then $S$ is maximally monotone of type (ANA)}.
\end{problem}
\section{More on the Fitzpatrick extension}\label{FITZCHARsec}
In this section, we suppose that $A$ is a closed monotone, quasidense subset of $E \times E^*$, and we give some characterizations of $A^\F$ in terms of marker functions.
\smallbreak  
From \Thm{RLMAXthm}(a), $A$ is maximally monotone.   If $b^* \in B^* = E^* \times E\dbs$ then, from \eqref{FSTAR}, \eqref{TH1}, and \eqref{FAT} applied to the Banach SN space $\big(B^*,\LT\big)$, 
\begin{equation}\label{FITZCHAR1}
\left.\begin{aligned}
{\Phi_A}^*&(b^*)
= \supn_{b \in E \times E^*}\big[\bra{b}{b^*} - \Phi_A(b)\big]\\
&= \supn_{b \in E \times E^*}\big[\Bra{Lb}{\LT b^*} - \Theta_A(Lb)\big]\\
&\le \supn_{d^* \in E^* \times E\dbs}\big[\Bra{d^*}{\LT b^*} - \Theta_A(d^*)\big] = {\Theta_A}^*\big(\LT b^*\big) = {\Theta_A}^@(b^*),
\end{aligned}\right\}
\end{equation}
and, from \eqref{FAT}, \eqref{PHI4} and \eqref{TH3},
\begin{equation}\label{FITZCHAR2}
\left.\begin{aligned}
{\Phi_A}\sqs(b^*) &= \supn_{d^* \in B^*}\big[\Bra{d^*}{\LT b^*} - {\Phi_A}^*(d^*)\big]\\
&\ge \supn_{a \in A}\big[\Bra{La}{\LT b^*} - {\Phi_A}^*(La)\big]\\
&= \supn_A\big[b^* - {\Phi_A}^@\big]
= \supn_A\big[b^* - q_L\big] = \Theta_A(b^*).
\end{aligned}\right\}
\end{equation}
It would have been impossible to state \eqref{FITZCHAR1} and \eqref{FITZCHAR2} in \Sec{PHIsec}, since  $B^*$ did not have a Banach SN structure in that section.
\par
\begin{lemma}\label{THPHTHlem}
We suppose that $g \in \PC(B^*)$ is a marker function for $A$.
Then $g \le {\Theta_A}^@$ and $g^@ \ge \Theta_A$ on $B^*$.
\end{lemma}
\begin{proof}
The first assertion is immediate from \eqref{MARK3} and \eqref{FITZCHAR1}.   Taking conjugates in \eqref{MARK3}, $g^@ \ge {\Phi_A}\sqs$ on $B^*$, and so the second assertion follows from \eqref{FITZCHAR2}.   In fact, the second assertion can be deduced from the first by taking conjugates since ${\Theta_A}^{@@} = \Theta_A$.   \big(See \cite[Lemma 4.3(c), p.\ 237]{SSDMON}.\big)
\end{proof}
\begin{theorem}[Invariance of coincidence sets for marker functions]\label{INVARthm}
Let\break $g \in \PC(B^*)$ be a marker function for $A$.   Then
\smallbreak
\centerline{$\big\{B^*|g^@ = q_\LT\big\} = \{B^* | g = \qLt\} = \big\{B^*|{\Theta_A} = q_\LT\big\}$.} 
\end{theorem}
\begin{proof}
From \Lem{THPHTHlem}, \eqref{GTH1}, \Thm{MARKthm} and \Lem{EESLlem},
\smallbreak
\centerline{${\Theta_A}^@ \ge g \ge \Theta_A \ge q_\LT\ \on\ B^*$.}
\smallbreak
\noindent
Thus $\big\{B^*|{\Theta_A}^@ = q_\LT\big\} \subset \big\{B^*|g = q_\LT\big\} \subset \big\{B^*|{\Theta_A} = q_\LT\big\}$.   If we apply \Lem{Llem}(c) to $f := \Theta_A$, we see that $\big\{B^*|{\Theta_A} = q_\LT\big\} \subset \big\{B^*|{\Theta_A}^@ = q_\LT\big\}$, and so $\big\{B^*|g = q_\LT\big\} = \big\{B^*|{\Theta_A} = q_\LT\big\}$.   Taking conjugates in \eqref{GTH1}, and using \Lem{THPHTHlem},  \Thm{MARKthm} and \Lem{EESLlem} again,
\smallbreak
\centerline{${\Theta_A}^@ \ge g^@ \ge \Theta_A \ge q_\LT\ \on\ B^*$.}
\smallbreak
\noindent
Using the same argument as above, $\big\{B^*|g^@ = q_\LT\big\} =  \big\{B^*|{\Theta_A} = q_\LT\big\}$.
\end{proof}
\begin{problem}
Is $g^@$ necessarily a marker function for $A$?
\end{problem}
\begin{theorem}[More characterizations of the Fitzpatrick extension]\label{FITZCHARthm}We have: 
\par
\noindent
{\rm(a)}\enspace $\big\{B^*|{\Phi_A}\sqs = q_\LT\big\} = A^\F = \big\{B^*|\Theta_A = q_\LT\big\}$.
\par
\noindent
{\rm(b)}\enspace If $g$ is a marker function for $A$ then $\big\{B^*|g^@ = q_\LT\big\} = \{B^* | g = \qLt\} = A^\F$.
 
\end{theorem}
\begin{proof}
From \Lem{PHSTMAlem}, ${\Phi_A}^*$ is a marker function for $A$, and so  \Thm{INVARthm} gives $\big\{B^*|{\Phi_A}\sqs = q_\LT\big\} = \big\{B^*|{\Phi_A}^* = q_\LT\big\} = \big\{B^*|{\Theta_A} = q_\LT\big\}$.  (a) follows since, from \eqref{FITZ1}, the middle of these sets is actually $A^\F$.   (b) follows from (a) and \Thm{INVARthm} as stated.
\end{proof} 
\begin{theorem}\label{AFMAXthm}
$A^\F$ is a maximally monotone subset of $E^* \times E\dbs$.
\end{theorem}
\begin{proof}
From \eqref{AFMON}, $A^\F$ is monotone.   Now suppose that $b^* \in E^* \times E\dbs$ and\quad $\inf\qLt\big(A^\F - b^*\big) \ge 0$.\quad   From \eqref{FITZ2}, $L(A) \subset A^F$, thus\quad $\inf\qLt\big(L(A) - b^*\big) \ge 0$,\quad and so, from \eqref{TH3} and the fact that, for all $a \in A$,\quad $\bra{a}{b^*} - q_L(a) - \qLt(b^*) = \Bra{La}{\LT b^*} - \qLt(La) - \qLt(b^*) = -\qLt(La - b^*)$,
$$\Theta_A(b^*) -  \qLt(b^*) = \supn_A\big[b^* - q_L - \qLt(b^*)\big] = -\inf\qLt\big(L(A) - b^*\big) \le 0.$$
Thus, from \Thm{MARKthm}, $\Theta_A(b^*) =  \qLt(b^*)$, and \Thm{FITZCHARthm}(a) implies that $b^* \in A^\F$.   This gives the desired result.  
\end{proof}
\begin{remark}\label{SQrem}
\Thm{FITZCHARthm}(a) implies that $(x^*,x\dbs) \in A^F$ exactly when $(x\dbs,x^*)$ is in the Gossez extension of $A$ \big(see \cite[Lemma~2.1,p. 275]{GOSSEZ}\big), which is known to be maximally monotone, so \Thm{AFMAXthm}  is to be expected.
\end{remark}
\begin{problem}
Is $A^\F$ necessarily an $\rLt$--dense subset of $E^* \times E\dbs$? 
\end{problem}
\section{On a result of Zagrodny}\label{ZAGsec}
We now give a generalization to Banach SN spaces of an inequality for\break monotone multifunctions proved by Zagrodny.   This generalization appears in \Thm{ZAGthm}; and in \Thm{ZAGEEthm}, we see how this result appears when applied to monotone multifunctions.   There is a discussion of Zagrodny's original result in \Rem{ZAGrem}.   The analysis in this section does not depend on any of the results in this paper after \Sec{LPOSsec} other than \Sec{EEsec}.   So let $(B,L)$ be a Banach SN space.   Then
\begin{equation}\label{XOTAsix}
\All\ d,e \in B,\ \|e\| \le \sqrt{2r_L(e) + 2r_L(d) - 2q_L(d - e)} + \|d\|.
\end{equation}
To see this, it suffices to observe that
\begin{align*}r_L(e) &+ r_L(d) - q_L(d - e) = r_L(e) + r_L(d) - q_L(e) - q_L(d) + \bra{d}{Le}\\
&= \half\|e\|^2 + \half\|d\|^2 + \bra{d}{Le} \ge \half\|e\|^2 + \half\|d\|^2 - \|d\|\|e\| = \half\big(\|e\| - \|d\|\big)^2.
\end{align*} 
\begin{lemma}\label{ZAGlem}
Let $A_0$ be an $L$--positive subset of $B$ and $e,d\in A_0$. Then
\smallbreak
\centerline{$\|e\| \le \sqrt{2r_L(e)} + \sqrt{2}\|d\| + \|d\| \le \sqrt{2r_L(e)} + \frac52\|d\|$.}
\end{lemma}
\begin{proof}
Since $A_0$ is $L$--positive,\quad $q_L(d - e) \ge 0$,\quad and so \eqref{XOTAsix} and \eqref{RLPOS} imply that 
\begin{align*}
\|e\| &\le \sqrt{2r_L(e) + 2r_L(d)} + \|d\|\\
&\le \sqrt{2r_L(e)} + \sqrt{2r_L(d)} + \|d\| \le \sqrt{2r_L(e)} + \sqrt{2}\|d\| + \|d\|.
\end{align*}
This gives the required result.
\end{proof}
\begin{theorem}\label{ZAGthm}
Let $A$ be an $L$--positive subset of $B$, $a \in A$ and $b \in B$.   Then\quad $\|a\| \le \sqrt{2r_L(a - b)} + {\ts\frac52}\dist(b,A) + \|b\|$.
\end{theorem}
\begin{proof}
Let $A_0$ be the $L$--positive set $A - b$.  Let $c \in A$.   Then $e := a - b \in A_0$ and $d := c - b \in A_0$.  From \Lem{ZAGlem},\quad $\|a - b\| \le \sqrt{2r_L(a - b)} + \frac52\|c - b\|$.\quad Taking the infimum over $c$,\quad $\|a - b\| \le \sqrt{2r_L(a - b)} + {\ts\frac52}\dist(b,A)$.
\end{proof}
\begin{theorem}\label{ZAGEEthm}
Let $A$ be a monotone subset of $E \times E^*$ and $(w,w^*) \in E \times E^*$.   Then there exists $M \ge 0$ such that, for all  $(s,s^*) \in A$,
$$\|(s,s^*)\| \le M + \sqrt{\|s - w\|^2 + \|s^* - w^*\|^2 + 2\bra{s - w}{s^* - w^*}}.$$
\end{theorem}
\begin{proof}
This follows from \Thm{ZAGthm}, with $M = \ts\frac52\dist\big((w,w^*),A\big) + \|(w,w^*)\|$.
\end{proof}
\begin{remark}\label{ZAGrem}
\Thm{ZAGEEthm} was motivated by (and clearly generalizes) the second assertion of Zagrodny, \cite[Corollary~3.4, pp.\ 780--781]{ZAGRODNY}, which is equivalent to the following:   {\rm Let $A$ be a maximally monotone subset of $E \times E^*$ of type (NI) and $(w,w^*) \in E \times E^*$.   Then there exist $\eps_0 > 0$ and $R > 0$ such that if $0 < \eps < \eps_0$, $(s,s^*) \in A$ and
$${\|s - w\|^2 + \|s^* - w^*\|^2 + 2\bra{s - w}{s^* - w^*}} \le \eps$$
then $\|(s,s^*)\| \le R$}.   \Thm{ZAGEEthm} shows that we only need to assume that $A$ is monotone, $\eps$ can be as large as we please, and $\|(s,s^*)\|$ is bounded by a function of the form $M + \sqrt\eps$.
\end{remark}

\end{document}               
Removed from RLMM95:
Appendix
Section 18
Section 17 after Problem 17.7
Section 16
Section 15
Section 14
Section 13
Section 12
\Rem{NIrem}